\documentclass[12pt]{amsart}

\usepackage[usenames,dvipsnames]{xcolor}
\definecolor{amethyst}{rgb}{0.6, 0.4, 0.8}

\usepackage[pagebackref=true, colorlinks, allcolors=amethyst]{hyperref}

\renewcommand*{\backref}[1]{}
\renewcommand*{\backrefalt}[4]{%
  \ifcase #1 %
    \relax
  \or
    $\uparrow$#2.%
  \else
    $\uparrow$#2.%
  \fi%
}

\usepackage{mathtools}

\usepackage{breakurl}
\usepackage{fullpage,url,amssymb,colonequals,algorithm2e}
\usepackage[shortlabels]{enumitem}

\usepackage{mathrsfs} 

\usepackage{tabularx} 
\usepackage{booktabs} 

\usepackage{tikz-cd,tikz}

\newcommand{\defi}[1]{\textsf{#1}} 

\usepackage{color}

\usepackage{soul}

\newcommand{\F}{\mathbb{F}}
\newcommand{\G}{\mathbb{G}}

\newcommand{\N}{\mathbb{N}}
\newcommand{\PP}{\mathbb{P}}
\newcommand{\Q}{\mathbb{Q}}

\newcommand{\Z}{\mathbb{Z}}
\newcommand{\Qbar}{{\overline{\Q}}}

\newcommand{\calL}{\mathcal{L}}
\newcommand{\calM}{\mathcal{M}}
\newcommand{\calN}{\mathcal{N}}
\newcommand{\calO}{\mathcal{O}}
\newcommand{\calP}{\mathcal{P}}

\renewcommand{\div}{\operatorname{Div}}
\DeclareMathOperator{\Div}{Div}

\DeclareMathOperator{\im}{im}

\DeclareMathOperator{\lcm}{lcm}

\DeclareMathOperator{\NS}{NS}

\DeclareMathOperator{\Norm}{Norm}

\DeclareMathOperator{\ord}{ord}
\DeclareMathOperator{\Pic}{Pic}

\DeclareMathOperator{\Proj}{Proj}

\DeclareMathOperator{\rk}{rk}

\DeclareMathOperator{\Spec}{Spec}

\DeclareMathOperator{\tr}{tr}

\DeclareMathOperator{\id}{id}
\renewcommand{\exp}{\operatorname{exp}}

\newcommand{\Coh}{\operatorname{Coh}}
\newcommand{\Geo}{\operatorname{Geo}}
\newcommand{\CQC}{\operatorname{Coh}}
\newcommand{\GQC}{\operatorname{Geo}}
\newcommand{\Nek}{\operatorname{Nek}}

\newcommand{\ol}{\overline}

\newcommand{\tensor}{\otimes} 

\newcommand{\Lc}{\mathcal{L}}

\newcommand{\dual}{\vee}
\newcommand{\sm}{\textnormal{sm}}
\newcommand*{\SheafIsom}{\mathrm{I}\kern -.5pt som}
\renewcommand{\phi}{\varphi}
\newcommand{\x}{\mathbf{x}}
\newcommand{\dR}{\textnormal{dR}}

\theoremstyle{plain}
\theoremstyle{plain}

\newtheorem{theorem}{Theorem}[section]
\newtheorem{lemma}[theorem]{Lemma}
\newtheorem{corollary}[theorem]{Corollary}
\newtheorem{proposition}[theorem]{Proposition}
\newtheorem{intheorem}{Theorem}

\theoremstyle{definition}
\newtheorem{definition}[theorem]{Definition}

\newtheorem{example}[theorem]{Example}

\newtheorem{myalgorithm}[theorem]{Algorithm}
\newtheorem{remark}[theorem]{Remark}

\newenvironment{enumalg}
{\begin{enumerate}}
{\end{enumerate}}

\usepackage{microtype}

\allowdisplaybreaks

\title[Geometric Quadratic Chabauty]{Geometric Quadratic Chabauty and $p$-adic heights}

\dedicatory{In memory of Bas Edixhoven.}

\thanks{JDR was supported by a grant from the Simons Foundation (550029, to John Voight).}
\thanks{SH was supported by National Science Foundation grant DGE-1840990.}
\thanks{PS was partially supported by NWO grant VI.Vidi.193.006.}

\author[Duque-Rosero]{Juanita Duque-Rosero}
\address{Juanita Duque-Rosero, Dartmouth College, 6188 Kemeny Hall, Hanover, NH 03755, USA}
\email{juanitaduquer@gmail.com}
\urladdr{\url{https://math.dartmouth.edu/~jduque/}}

\author[Hashimoto]{Sachi Hashimoto}
\address{Sachi Hashimoto, Max-Planck-Institut f\"ur Mathematik
in den Naturwissenschaften, Inselstra{\ss}e 22, 04103 Leipzig, Germany}
\email{sachi.hashimoto@mis.mpg.de}
\urladdr{\url{https://sachihashimoto.github.io}}

\author[Spelier]{Pim Spelier}
\address{Pim Spelier, Mathematisch Instituut, Universiteit Leiden, Postbus 9512, 2300 RA Leiden, Netherlands}
\email{spelierp@math.leidenuniv.nl}
\urladdr{\url{https://sites.google.com/view/pim-spelier/home}}

\begin{document}

\begin{abstract}
Let $X$ be a curve of genus $g>1$ over $\Q$ whose Jacobian $J$ has Mordell--Weil rank $r$ and N\'eron--Severi rank $\rho$. 
When $r < g+ \rho - 1$, the geometric quadratic Chabauty method determines a finite set of $p$-adic points containing the rational points of $X$. 
We describe algorithms for geometric quadratic Chabauty that translate the geometric quadratic Chabauty method into the language of $p$-adic heights and $p$-adic (Coleman) integrals. This translation also allows us to give a comparison to the (original) cohomological method for quadratic Chabauty. We show that the finite set of $p$-adic points produced by the geometric method is contained in the finite set produced by the cohomological method, and give a description of their difference. 
\end{abstract}
\maketitle 
\date{\today}

\section{Introduction}
Let $X_\Q$ be a smooth, projective, geometrically irreducible curve of genus $g>1$ over $\Q$. The problem of describing $X_\Q(\Q)$, the set of rational points of $X_{\Q}$, has fascinated mathematicians for centuries. A famous conjecture of Mordell \cite{Mordell} is that, for $g>1$, the set $X_{\Q}(\Q)$ is finite. Faltings's theorem states that Mordell's conjecture is true \cite{Faltings}. However, Faltings's theorem is not effective, meaning that it does not give a method to determine the set of rational points. There is still an ongoing effort to find explicit methods to compute the set $X_\Q(\Q)$. Chabauty's theorem \cite{Chabauty} gives a finiteness result for $X_\Q(\Q)$ on certain curves by using $p$-adic analysis. This was made effective by Coleman \cite{Coleman} through the development of Coleman integration; he gave a method to find $p$-adic power series that vanish on a superset of $X_\Q(\Q)$ for the curves Chabauty considered. This breakthrough is the starting point for the Chabauty--Kim program \cite{KimSelmerVarieties} of $p$-adic methods for proving the finiteness of $X_\Q(\Q)$ generalizing Chabauty and Coleman's method. The quadratic Chabauty method \cite{QCIntegral,QCI,QCII,EdixhovenLido,QCarakelov}  is an effective instance of the Chabauty--Kim method, first developed by Balakrishnan and Dogra, for studying the rational points of $X_\Q$.

Let $J_\Q$ be the Jacobian of $X_\Q$, with Mordell--Weil rank $r$ and N\'eron--Severi rank $\rho \colonequals \rk \NS(J_\Q) >1$. Let $p>2$ be a prime, not necessarily of good reduction for $X_\Q$. Quadratic Chabauty is an effective $p$-adic method for producing a finite set of $p$-adic points containing the rational points of $X_\Q$, when $r<g+ \rho - 1$. There are several approaches to the quadratic Chabauty method. The (original) cohomological quadratic Chabauty method \cite{QCI,QCII} studies $X_\Q(\Q)$ using $p$-adic height functions and works in certain Selmer varieties (for $p$ of good reduction). This method has been applied to determine the rational points on many modular curves \cite{RecentApproaches,examplesandalg}, including the cursed curve \cite{QCCartan}, a famously difficult problem. The geometric quadratic Chabauty method \cite{EdixhovenLido} is an algebro-geometric method for quadratic Chabauty, and the computations take place in $\G_m$-torsors over $J_\Q$.

In this paper, we give a comparison of the geometric and cohomological methods for quadratic Chabauty in the cases where both methods can be applied. We prove the following theorem.
\begin{intheorem}[Comparison Theorem (Theorem~\ref{thm:comparison})]\label{the:mainTheorem}
Assume that $p$ is a prime of good reduction for $X_{\Q}$. Assume that $r = g$, $\rho > 1$, and the $p$-adic closure $\overline{J_\Q(\Q)}$ is of finite index in $J_\Q(\Q_p) $. Assume $X_\Q(\Q) \neq \emptyset$, and let $b \in X_\Q(\Q)$ be a choice of a rational base point. Let $X(\Q_p)_{\Coh}$ be the finite set of $p$-adic points obtained under these assumptions with the cohomological quadratic Chabauty method (see Definition~\ref{def:qcset} and Remark~\ref{rem:definitionqcoh}). Let $X/\Z$ be a proper regular model of $X_\Q$.
Let $X(\Z_p)_{\Geo}$ be the finite set of $p$-adic points obtained with the geometric quadratic Chabauty method (see Definition~\ref{def:xzpgeo}). Then we have the inclusions
\[
X_\Q(\Q) \subseteq X(\Z_p)_{\Geo} \subseteq X(\Q_p)_{\Coh} \subseteq X_\Q(\Q_p),
\]
and we can explicitly characterize $X(\Q_p)_{\Coh} \setminus X(\Z_p)_{\Geo}$.
\end{intheorem} 

In \cite{geometricLinearChab}, it is shown that the classical Chabauty--Coleman method \cite{EffectiveChab} and the geometric linear Chabauty method \cite{PimThesis} are related by a similar comparison theorem.

The geometric quadratic Chabauty method studies the Poincar\'e torsor, the universal $\G_m$-biextension over $J_\Q \times J_\Q$.
By pulling back the Poincar\'e torsor by a nontrivial trace zero morphism $f\colon  J_\Q \to J_\Q$, we can construct a nontrivial torsor $T$ over the N\'eron model of $J_\Q$ whose restriction to $X_\Q$ is trivial. This allows us to embed $X_\Q$ into $T$ through a section. The idea of the geometric quadratic Chabauty method is to intersect the image of the integer points on a regular model of $X_\Q$ with the $p$-adic closure of the integer points $\overline{T(\Z)}$. This intersection contains $X_\Q(\Q)$. 

Suppose further that $p$ is a prime of good reduction for $X_\Q$. We give new algorithms for geometric quadratic Chabauty that work mainly in the trivial biextension $\Q_p^g \times \Q_p^g \times \Q_p$. Working on the trivial biextension translates the geometric quadratic Chabauty method into the language of Coleman--Gross heights \cite{ColemanGross} and Coleman integrals \cite{Coleman}. The main contribution of this paper is to explicitly give this translation into the language of heights and Coleman integrals. This translation allows us to prove the comparison theorem between the cohomological quadratic Chabauty method and the geometric quadratic Chabauty method.  We also give an algorithm to compute the local heights away from $p$ associated to the curve $X_\Q$. These heights are also studied in \cite{BettsDogra}.

We further leverage the language of $p$-adic heights to compute the embedding of $X_\Q$ into $T$ and the integer points $\overline{T(\Z)}$ as convergent power series. Then determining up to finite $p$-adic precision a finite set containing $X_\Q(\Q)$ reduces to solving simple polynomial equations.
Theoretically, by working modulo $p^k$ for large enough $k \in \N$, the geometric quadratic Chabauty method will always produce a finite set of $p$-adic points with precision $k$ containing $X_\Q(\Q)$. We describe algorithms for finding this finite set of $p$-adic points that are practical when $X_\Q$ is a hyperelliptic curve. Our \texttt{Magma} code implementing these algorithms can be found in \cite{DRHSgitrepo}.

Finally, we present an example of our new method applied to the modular curve $X_0(67)^+$ and a trace zero endomorphism $f$ arising from the Hecke operator $T_2$. Even though the rational points on this curve have already been determined \cite{RecentApproaches}, this provides a new way of analyzing the set of rational points. 

\section{Overview and Set-up}
\label{sec:overview}

We first set up some notation and give a broad overview of the geometric quadratic Chabauty method, then outline the contents of our paper.

Let $X_\Q$ be any smooth, projective, geometrically irreducible curve over $\Q$ with a proper regular model $X$ of $X_\Q$ over the integers and a fixed base point $b \in X_\Q(\Q) = X(\Z)$.  Let $X^\sm$ denote the open subscheme of $X$ consisting of points at which $X$ is smooth over $\Z$; then $X^\sm(\Z) = X(\Z)$.  Let $J_\Q$ denote the Jacobian of $X_\Q$ and $J$ denote the N\'eron model of $J_\Q$ over the integers. Suppose $J_\Q$ has Mordell--Weil rank $r$ and N\'eron--Severi rank $\rho=\rho(J_\Q)$. Let $p$ be a prime greater than $2$ not necessarily of good reduction for $X_\Q$.

The goal in geometric quadratic Chabauty is to lift $X$ into a non-trivial $\G_m^{\rho-1}$-torsor $T$ over $J$ through a section $\widetilde{j_b}$ lying over the Abel--Jacobi embedding $j_b\colon  X^\sm \to J$. Over $\Q$ we find this section $\widetilde{j_b}$ by giving a trivializing section of the $\G_m^{\rho -1}$-torsor $j_b^* T_\Q$ over $X_\Q$. If we want to spread this out over $\Z$, there is an obstruction coming from the multidegree.
\begin{definition}
The \defi{multidegree} of a line bundle $\calL$ on a curve $C$ with geometrically irreducible components $(C_i)_{i \in I}$ over $\Qbar$ is $(\deg \calL|_{C_i})_{i \in I}$.
\end{definition}
The map $\Pic(X) \to \Pic(X_\Q)$ is not in general an isomorphism, and $j_b^* T$ is not in general trivial over $X$ since its multidegree over the fibers $X_{\F_{\ell}}$ of $X$ might be non-zero. This is the only obstruction: the torsor can be trivialized over an open $U \subset X^\sm$ constructed by picking one geometrically irreducible component in each fiber $X_{\F_{\ell}}$ and removing the other irreducible components. If $U$ does not contain a $\F_q$-point of $X$, then it certainly contains no integer points, and hence we only consider $U$ that are everywhere locally soluble. We call these everywhere locally solvable fiberwise geometrically irreducible open $U \subset X^\sm$ \defi{simple open sets}.
By \cite[\href{https://stacks.math.columbia.edu/tag/04KV}{Tag 04KV}]{stacks-project} every irreducible component of $X_{\F_{\ell}}$ admitting a smooth $\F_\ell$-point is geometrically irreducible. Hence every point $P \in X^\sm(\Z)$ is contained in $U(\Z)$ for a unique simple open $U$. 
There is a finite number of simple open sets. For every such open, the map $\Pic(U) \to \Pic(X_\Q)$ is an isomorphism. We fix a simple open $U$, and obtain a trivialization $\widetilde{j_b}\colon  U \to T$ lying over $j_b$.

Because $\G_m(\Z) = \{\pm 1\}$ is finite, we can expect the closure of $T(\Z)$ inside the $(g+ \rho-1)$-dimensional $p$-adic manifold $T(\Z_p)$ to be of dimension at most $r$. The image of the $p$-adic points of $U$, namely $\widetilde{j_b}(U(\Z_p))$, is of dimension $1$. Given this $T$, we see the analogue of the classical Chabauty's theorem, that applies for curves satisfying the inequality $r<g$ \cite{Chabauty}.
\begin{theorem}\cite[Section~9.2]{EdixhovenLido}
When $r < g + \rho - 1$, the intersection 
\[\widetilde{j_b}(U(\Z_p)) \cap \overline{T(\Z)} \subset T(\Z_p)\] is finite.
\end{theorem}
\begin{definition}
\label{def:xzpgeo}
The \defi{geometric quadratic Chabauty set} $X(\Z_p)_{\Geo}$ is defined to be the union over the simple open sets $U$ of $ \widetilde{j_b}^*(\widetilde{j_b}(U(\Z_p)) \cap \overline{T(\Z)}) \subset U(\Z_p) \subset X(\Z_p)$.
\end{definition}
The geometric quadratic Chabauty method computes this finite set $X(\Z_p)_{\Geo}$, working in one simple open $U \subset X$ and one residue disk of $U(\Z_p)$ at a time. In Algorithm~\ref{alg:mainGQCalg} we give an algorithm to determine $\widetilde{j_b}(U(\Z_p)) \cap \overline{T(\Z)}$ to finite precision.

To construct the $\G_m^{\rho -1}$-torsor $T$ over $J$ we start with the universal $\G_m$-torsor. In our calculations this takes the form of the Poincar\'e torsor $\calM^\times$ over $J \times J^0$ (this is actually a pullback of the Poincar\'e torsor over $J \times J^{\dual 0}$; for more details see Section~\ref{sec:biextension}). Here $J^{\vee 0}$ is the fiberwise connected component of $J^\vee$ containing $0$. 
\begin{remark}
\label{rem:J0}
When $p$ is a prime of good reduction for $X$, we have $J^{ 0}_{\Z_{(p)}} = J_{\Z_{(p)}}$ and $J^{ \vee0}_{\Z_{(p)}} = J^\vee_{\Z_{(p)}}$.
\end{remark}
By the universality of $\calM^\times$, we want to construct $T$ by pulling back $\calM^\times$ along morphisms $(\id,\alpha_i)\colon  J \to J \times J^0$ for $i= 1, \dots, \rho -1$. Define 
\begin{align}
\label{eq:defm}
m \colonequals \lcm \{ \exp\left((J/J^0)(\overline{\F}_q)\right) \mid q \text{ prime}\},
\end{align} where $\exp(G) \in \N_{\geq 1}$ is the exponent of a finite group $G$. Note that $m\cdot \colon  J \to J^0$ is then a well defined morphism. Any morphism of schemes $J\to J$ can be written as a translation composed with an endomorphism, and hence we choose our morphisms $\alpha_i\colon  J \to J^0$ to be of the form $m\cdot \circ \tr_{c_i} \circ f_i$ with $c_i \in J(\Z)$ and $f_i\colon J\to J$ a morphism of group schemes. 

The torsor $T$ is the product $T = \prod_{i =1}^{\rho-1}(\id, \alpha_i)^* \calM^\times$ as a fiber product over $J$.  We also let $\calM^{\times, \rho-1}$ be the product taken as a fiber product over $J$ via the first projection map $\calM^\times\to J\times J^0\to J$. 
In order to embed $U$ through a section $\widetilde{j_b}\colon  U \to T$, the torsor $T$ pulled back to $U$ must be trivial: that is $j_b^* (\id, \alpha_i)^* \calM^\times$ must be trivial over $U$.
The torsor $(\id, \alpha_i)^* \calM^\times$ over $J$ can be thought of as the total space of a line bundle without its zero section, and the condition that its pullback $L_{\alpha_i} \colonequals j_b^*(\id, \alpha_i)^* \calM^\times$ to $U$ is trivial forces the corresponding line bundle to be degree $0$. Equivalently, the trace of $f_i$ must be $0$. The condition that $L_{\alpha_i}$ is trivial uniquely determines $c_i$.  
\begin{equation}
\begin{tikzcd}
                                                 & T \arrow[rr] \arrow[d]                                       &  & {\calM^{\times, \rho-1}} \arrow[d] \\
U \arrow[r, "j_b"] \arrow[ru, "\widetilde{j_b}"] & J \arrow[rr, "{(\id, m \cdot \circ \tr_{c_i} \circ f_i)_i}"] &  & J\times (J^{ 0})^{\rho-1}     
\end{tikzcd}
\end{equation}
Because the N\'eron--Severi rank of $J_\Q$ is $\rho$, the Jacobian  $J$ has $\rho-1$ independent non-trivial endomorphisms of trace zero. 

\begin{definition}
For $Y$ a scheme, $S$ a ring with residue field $\Spec \F_p \to \Spec S$ and $Q \in Y(\F_p)$, we define the  \defi{residue disk over $Q$}, denoted by  $Y(S)_Q \colonequals  \{ y \in  Y(S) \mid \overline{y} = Q \} $, to be the set of all $S$-points specializing to $Q$.
\end{definition}

Let  $\overline{P} \in U(\F_p)$.
The residue disk $U(\Z_p)_{\overline{P}}$ embeds into the residue disk $T(\Z_p)_{\widetilde{j_b}(\overline{P})}$ of $T$ through the section $\widetilde{j_b}$. Since $p > 2$, we have that $1$ and $-1$ reduce to different points modulo $p$ and hence the map $T(\Z)_{\widetilde{j_b}(\overline{P})} \to J(\Z)_{j_b(\overline{P})}$ is a bijection. By \cite[Proposition~2.3]{parent} and the fact that $p > 2$ the residue disk $J(\Z)_{j_b(\overline{P})}$ is up to a translation isomorphic to $\Z_p^r$. In \cite[Theorem~4.10]{EdixhovenLido} this bijection $T(\Z)_{\widetilde{j_b}(\overline{P})} \to J(\Z)_{j_b(\overline{P})}$ is upgraded to a morphism $\kappa\colon  \Z_p^r \to T(\Z_p)_{\widetilde{j_b}(\overline{P})}$ with image exactly $\overline{T(\Z)}_{ \widetilde{j_b}(\overline{P}) }$. 

In this paper we make the geometric quadratic Chabauty method explicit in the case where $p$ is of good reduction by giving algorithms to compute $\widetilde{j_b}$ and $\kappa$ in a residue disk as polynomials in parameters up to finite precision. This translates the geometric Chabauty method into solving simple polynomial equations. We also give algorithms to work in residue disks of $T$ explicitly using $p$-adic heights and Coleman integrals. Moreover, by writing the geometric quadratic Chabauty method in terms of $p$-adic heights and Coleman integrals, we are able to prove Theorem~\ref{the:mainTheorem}.

\subsection{Structure of the paper}

In Section~\ref{sec:biextension} we provide background on the Poincar\'e torsor and its realizations. We solve the problem of how to efficiently represent elements of a residue disk of $T$. We show how to represent elements of the Poincar\'e torsor $\calM^\times$ using the following statement that appears in \cite[Section~9.3]{EdixhovenLido}.
\begin{proposition}
Let $p>2$ be a prime of good reduction for $X$. There is a morphism of biextensions over $J(\Z_p) \times J(\Z_p)$
\begin{align}
\Psi\colon  \calM^\times(\Z_p) \to J(\Z_p) \times J(\Z_p) \times \Q_p,
\end{align}
with the trivial $\Q_p$-biextension structure on the latter product. 
\end{proposition}
By Remark~\ref{rem:J0}, we have that $J^0(\Z_p)=J(\Z_p)$.
This proposition allows us to record elements of $\calM^\times(\Z_p)$ up to $p$-adic finite precision. 
In Proposition~\ref{prop:intpointscalN} we describe the image of integer points of $T$ in this trivial biextension $\calN \colonequals J(\Z_p) \times J(\Z_p) \times \Q_p$. 

Since we can construct a bijection from residue disks of $J(\Z_p)$ to $\Z_p^g$ using Coleman integrals, we can explicitly write down a homeomorphism from the residue disk $T(\Z_p)_{\widetilde{j_b}(\overline{P})}$ to $\Z_p^g \times \Q_p^{\rho - 1}$ factoring through $\Psi$; this is done in Corollary~\ref{lem:paramT}. Crucially, we prove that this homeomorphism is given by convergent power series on $\Z_p^{g+ \rho - 1}$, i.e. power series that modulo every power of $p$ are given by polynomials. 

Then in Section~\ref{sec:linebundle} we give an algorithm to construct the unique line bundle associated to the endomorphism $f$ from a divisor in $U \times X$ satisfying certain properties described in Lemma~\ref{lemma:A-alpha}. Using this line bundle we write down a theoretical formula for the trivializing section $\widetilde{j_b}\colon  U \to T$. 
We give an algorithm for computing the convergent power series describing the embedding of a residue disk of the curve into the biextension $\calN$ in Section~\ref{sec:embeddingcurve}. In Section~\ref{sec:integerpoints} we give formulas for computing points in the biextension $\calN$ that are the image of generating sections of certain residue disks of $\calM$.

In Section~\ref{sec:upperbound} we tie everything together with the algorithm for geometric quadratic Chabauty in a residue disk $U(\Z)_{\overline{P}}$. In this section, we also describe how to compute a finite set of $p$-adic points to finite precision containing the integer points in a single residue disk $U(\Z)_{\overline{P}}$. We do this by reducing our computations to $T(\Z/p^k\Z)_{\widetilde{j_b}(\overline{P})}$ and using a Hensel-like lemma \cite[Theorem~4.12]{EdixhovenLido}. By iterating over residue disks we find $X(\Z_p)_{\Geo}$ up to finite precision. 

The comparison theorem appears in Section~\ref{sec:comparison}. Theorem~\ref{thm:comparison} states that the finite set of points found by the cohomological quadratic Chabauty method is a superset of the points found by the geometric method, and gives an explicit description of the points in their difference.

Section~\ref{sec:example} shows a worked example of the algorithms applied to the case of $X_0(67)^+$. The rational points on this curve have been determined previously \cite{RecentApproaches}, but the computations here demonstrate the practicality of the geometric quadratic Chabauty algorithms presented here for hyperelliptic modular curves.

\section{Understanding the biextension and \texorpdfstring{$T$}{T}}
\label{sec:biextension}

A crucial object of study in our paper is the Poincar\'e torsor. This has four incarnations, which we introduce in the following four subsections. 
Sections \ref{sec:poincaretorsor} and \ref{subsec:M} are expository sections and introduce important background from \cite{EdixhovenLido}. Section~\ref{subsec:N} introduces the trivial biextension, and contains new propositions relating the biextension to $p$-adic heights. Section~\ref{subsec:Tf} introduces the pseudoparametrization of the torsor that we work with for the rest of the paper, and proves that the pseudoparametrization is given by nice convergent power series modulo powers of $p$.

\subsection{The Poincar\'e torsor \texorpdfstring{$\calP$}{P}}
\label{sec:poincaretorsor}

First we introduce the Poincar\'e torsor $\calP^\times_\Q$ over $J_\Q \times J^\dual_\Q$, its biextension structure, and the torsor $\calP^\times$ over the integers. For more details on the Poincar\'e torsor and biextensions, see \cite[\S I.2.5]{MoretBailly} or Grothendieck's Expos\'es  VII and VIII \cite{SGA7-I}.
The abelian variety $J_\Q^\dual$ is a moduli space for line bundles algebraically equivalent to zero on $J_\Q$; every $[c] \in J_\Q^\vee$ corresponds to a line bundle $\calL_c$ on $J_\Q$. The universal line bundle over $J_\Q \times J_\Q^\dual$ is the \defi{Poincar\'e bundle} $\calP_\Q$. It satisfies the property that $\calP_{\Q}|_{J_\Q \times {[c]}} \simeq \calL_c$ and it is rigidified at $0$, i.e.  $\calP_\Q|_{ {0} \times J^\dual}$ is trivial. Furthermore, under the natural identification $(J_\Q^\vee)^\vee = J_\Q$, this line bundle is also the universal line bundle over $J_\Q \times J_\Q^\dual$ parametrizing line bundles on $J_\Q^\dual$.

Given a line bundle $\calL$ over a scheme $S$, there is an associated $\G_m$-torsor $\calL^\times$ defined by taking the sheaf of non-vanishing sections, and similarly given a $\G_m$-torsor $Y$ there is an associated line bundle $Y \tensor_{\calO_{S}^\times} \calO_S$. Applying these associations to the Poincar\'e bundle, we obtain the universal $\G_m$-torsor $\calP_\Q^\times$ over $J_\Q \times J_\Q^\dual$, called the \defi{Poincar\'e torsor}. Alternatively, 
\begin{align*}
\calP_\Q^\times = \mathrm{Isom}_{J_\Q \times J_\Q^\dual}(\calO_{J_\Q \times J_\Q^\dual},\calP_\Q),
\end{align*}
i.e. for a scheme $S/(J_\Q \times J_\Q^\dual)$ we have that $\calP_\Q^\times(S)$ consists of isomorphisms of line bundles $\calO_S \to (\calP_\Q)_S$. This set $\calP_\Q^\times(S)$ is an $\calO_S(S)^\times$-pseudotorsor: either empty or an $\calO_S(S)^\times$-torsor.

The Poincar\'e torsor $\calP_\Q^\times$ has the structure of a \defi{biextension} over $J_\Q \times J_\Q^\dual$, as we will now explain. Addition in $J_\Q^\dual$ corresponds to tensoring line bundles on $J_\Q$. This, along with the theorem of the square, induces a partial group law on $\calP_\Q^\times$. Let $S$ be a scheme over $\Q$. For $x \in J_\Q(S)$ and $y_1,y_2 \in J_\Q^\dual(S)$ we have a tensor product which is an isomorphism of $\G_m$-torsors
\begin{align*}
(x,y_1)^* \calP_\Q^\times \tensor (x,y_2)^* \calP_\Q^\times \to (x,y_1 + y_2)^* \calP_\Q^\times
\end{align*}
that we denote by $\tensor_2$, because we are adding on the second coordinate (while the first coordinate stays fixed). Similarly since $(J_\Q^\dual)^\dual$ is canonically identified with $J_\Q$, we also have the tensor product
\begin{align*}
 (x_1,y)^* \calP_\Q^\times \tensor (x_2,y)^* \calP_\Q^\times \to (x_1 + x_2,y)^* \calP_\Q^\times
\end{align*}
  called $\tensor_1$.
These two partial group laws are compatible.  Let $x_1,x_2 \in J_\Q(S)$,  $y_1,y_2 \in J_\Q^\dual(S)$, and $z_{ij} \in (x_i,y_j)^*\calP_\Q^\times(S)$, for $i,j \in \{1,2\}$. Then 
\begin{align*}
(z_{11} \tensor_2 z_{12}) \tensor_1 (z_{21} \tensor_2 z_{22}) = (z_{11} \tensor_1 z_{21}) \tensor_2 (z_{12} \tensor_1 z_{22}).
\end{align*}
 In other words, tensoring points in the biextension is not order-dependent. The structure of these two partial group laws over the product $J_\Q \times J_\Q^\dual$, together with this compatibility, makes $\calP_\Q^\times$ a $\G_m$-biextension over $J_\Q \times J_\Q^\dual$.

For our applications, we need to work over the integers. Let $J^0$ be the fiberwise connected component of $J$ containing $0$. 
This represents line bundles on $C$ that are fiberwise of multidegree $0$.  Let $J^\dual$ be the N\'eron model of $J_\Q^\dual$ and similarly let $J^{\dual 0}$ be the fiberwise connected component of $J^\dual$ containing $0$. 
The Poincar\'e torsor extends to a biextension $\calP^\times$ over $J \times J^{\dual 0}$. In particular, the integer points of $\calP^\times$ lying over $(x,y) \in (J \times J^{\dual 0})(\Z)$ form a $\G_m(\Z)$-torsor, i.e. a $\{\pm 1\}$-torsor.  So there is exactly one integer point lying over $(x,y)$, up to sign.

\subsection{The biextension \texorpdfstring{$\calM$}{M}}
\label{subsec:M}
To work with explicit computations of points in the Poincar\'e torsor in practice, we need a few modifications of $\calP^\times$. We introduce two torsors over $J \times J^0$, $\calM^\times$ and $\calN$ the trivial biextension.

We first discuss the construction of $\calM^\times$ and the generating sections of its residue disks.
The Abel--Jacobi embedding induces an isomorphism $j_b^*\colon  J^{\dual} \to J$ and hence an isomorphism $j_b^*\colon  J^{\dual 0} \to J^0$.
We define
\begin{align}
\calM^\times \colonequals (\id,j_b^{*,-1})^* \calP^\times .
\end{align}
For the torsor $\calM^\times$, we have an explicit description of the fibers. Let $S$ be a scheme, $x \in J(S)$ be a point corresponding to a line bundle $\calL$, and $y \in J^0(S)$ be a point with representing divisor $E = E^+ - E^-$ such that $E^+$ and $E^-$ are effective and of the same multidegree. We denote the fiber $(x,y)^*\calM^\times$ of $\calM^\times$ over $(x,y)\in (J\times J^0)(S)$ by $\calM^\times(x,y)$. This fiber $\calM^\times(x,y)$ is the $\G_m$-torsor
\begin{align}
\label{eqn:fiber}
E^* \calL^\times \colonequals \Norm_{E^+/S} \big(\calL^\times\mid_{E^+}\big) \tensor \Norm_{E^-/S} \big(\calL^\times\mid_{E^-}\big)^{-1},
\end{align}
which we also denote by $\Norm_{E/S} \calL^\times$. When $S = \Spec \Z$ we also write simply $\Norm_{E} \calL^\times$.   This fiber can be thought of as the aggregate of how $\calL$ looks around $E$.

This description of the fiber is proven in \cite[Proposition~6.8.7]{EdixhovenLido} and more general facts about these norms can be found in \cite[Section~6]{EdixhovenLido}. Because equation \eqref{eqn:fiber} may seem a bit opaque, we provide some examples of how to apply the formula in practice.

\begin{definition}
Let $S$ be a scheme. Let $D$ and $E$ be two relative Cartier divisors on $X_S/S$. We say $D$ and $E$ are \defi{disjoint over $S$} if their support is disjoint as closed subschemes of $X_S$. In particular, it is not enough to have disjoint $S$-points if $D$ or $E$ does not split completely over $S$.
\end{definition}

\begin{example}
\label{ex:disjointrep}
Let $S$ be a scheme, $[D] \in J(S),$ and $[E] \in J^0(S)$ be points of $J$ and $J^0$ with representing divisors $D$ and $E$ where $E$ has multidegree $0$. Assume $D$ and $E$ are disjoint over $S$, and write $E = E^+ - E^-$ with $E^+,E^-$ effective. Then the $\G_m$-torsor $E^* \calO_X(D)^\times$ is generated by $\Norm_{E^+/S}(1) \tensor \Norm_{E^-/S}(1)^{-1}$ where $1$ is here seen as a section of $\calO_X(D)^\times|_{E^{\pm 1}}$. We also denote this generator by $E^* 1$. 
\end{example}

\begin{example}
Suppose the fiber of $X^{\sm}/\Z$ over $2$ is geometrically irreducible.
Let $[D] \in J(\Z)$ and $[E] \in J^0(\Z)$ be points of $J$ and $J^0$ with representing divisors $D$ and $E$. Assume $D$ and $E$ are disjoint over $\Z[\frac12]$ and meet with multiplicity $1$ over $2$. Then $E^* \calO_X(D)^\times$ is generated by $2^{-1}E^*1$. 
\end{example}

\begin{remark}
\label{rem:divGgen}
Let $S$ be a scheme.
If $D = \div g \in \Div^0(X_S/S)$ is the principal divisor of a rational function $g$ and is disjoint from $E \in \Div^0(X_S/S)$, then the isomorphism  $\calO_X(D) \to \calO_X$ given by multiplication by $g$ induces an isomorphism $E^* \calO_X(D)^\times \to E^* \calO_X^\times$ sending $E^* 1$ to $E^* g(E)$ where $g(E) \in \G_m(S)$.
\end{remark}

\begin{remark}
In general, if $[D] \in J(\Z)$, $[E] \in J^0(\Z)$, and we have a choice of representing divisors $D$ and $E$ that are disjoint over $\Q$, using intersection theory we can determine $n \in \Q^{\times}$ 
 unique up to sign, such that $\Norm_E \calO_X(D)^\times$ is generated by $ n \cdot E^* 1$. 
  If $E$ is not of multidegree $0$, there is a unique vertical divisor $V \subset C$ with $V+E$ of multidegree $0$. In this case, one can compute the unique rational number $a$ up to sign such that $(E+V)^* \calO_X(D)^\times = a \Norm_{E} \calO_X(D)^\times$. This is treated in detail in \cite[Section~6.9]{EdixhovenLido}.
\end{remark}

The partial group laws on $\calM^\times$ are also very explicit: let $[E], [E_1], [E_2] \in J^0(S)$ and $\calL, \calL_1, \calL_2 \in J(S)$.  They are given by the morphisms 
\begin{align}
\label{eqn:tensor2onM}
E_1^* \calL^\times \tensor E_2^* \calL^\times \to (E_1+E_2)^* \calL^\times
\end{align} corresponding to $\otimes_2$ and  
\begin{align}
\label{eqn:tensor1onM}
E^* \calL_1^\times \tensor E^* \calL_2^\times \to E^* (\calL_1 \tensor \calL_2)^\times\end{align}  corresponding to $\otimes_1$.

\begin{example}
\label{example:addpointsM}
Let $x_1,x_2 \in J(\Z)$ and $y_1,y_2 \in J^0(\Z)$. Let $z_{ij} \in \calM^\times(\Z)$ be points above $(x_i,y_j)$ for $i\in\{1,2\}$. Then for $n_1,n_2,m_1,m_2\in\Z$ we can construct points above $(n_1 x_1 + n_2 x_2, m_1 y_1 + m_2 y_2)$ by the formula
\begin{align*}
\left(z_{11}^{\tensor_2 m_1} \tensor_2 z_{12}^{\tensor_2 m_2}\right)^{\tensor_1 n_1} \tensor_1 \left(z_{21}^{\tensor_2 m_1} \tensor_2 z_{22}^{\tensor_2 m_2}\right)^{\tensor_1 n_2}.
\end{align*}
\end{example}

This allows us to construct many integer points of $\calM^\times$ by starting with a few points that lie over generators of the Jacobian and then applying the partial group laws. In Section~\ref{sec:integerpoints} we will use this idea to determine the integer points of the torsor $T$ landing in a specific residue disk of $T$.

\subsection{The trivial biextension \texorpdfstring{$\calN$}{N}}
\label{subsec:N}
In practice, we will often translate between $\calM$ and the trivial biextension $\calN$ where we do our computations. We explain how to make this translation following \cite[Section~9.3]{EdixhovenLido}. From now on, we assume $p>2$ is a prime of good reduction for $X_\Q$.

Let $[D] \in J(\Q_p)$ and $[E] \in J^0(\Q_p)$ be divisor classes with a choice of representing divisors $D$ and $E$ that are disjoint over $\Q_p$. Then $E^* \calO_X(D)^\times$ is a $\Q_p^\times$-torsor, trivial with generator $E^* 1$ by Example~\ref{ex:disjointrep}. Let $h_p$ be the cyclotomic Coleman--Gross local height at $p$ with respect to an isotropic splitting  $H^1_{\dR}(X) = H^0(X,\Omega_X^1) \oplus W$ of the Hodge filtration \cite[Section~5]{ColemanGross}. Choose a branch of the logarithm with $\log p = 0$ so that it is compatible with $h_p$. The height $h_p$ is a biadditive, symmetric pairing on disjoint divisors of degree $0$, taking values in $\Q_p$. For $f$ a rational function and $\Div f$ its associated divisor, it also satisfies the equality $h_p(D, \Div f) = \log f(D)$. 

\begin{remark}
The assumption that $p$ is a prime of good reduction for $X$ is used to define the logarithm of $J_{\Z_p}$, and to compute the Coleman--Gross height and iterated Coleman integrals. There is a more general construction using Vologodsky integrals to construct the Coleman--Gross height \cite{BesserHeightsVologodsky}, but currently there is no known way to compute this more general height for a prime of bad reduction. 
\end{remark}

We define a map
\begin{align}
\label{def:calN}
\psi\colon  \calM^\times(\Z_p) &\to \Q_p \\\notag
E^* \lambda \in E^* \calO_X(D)^\times &\mapsto \log \lambda + h_p(D,E).
\end{align} 
We define $\calN$ to be the trivial $\Q_p$-biextension $J(\Q_p) \times J(\Q_p) \times \Q_p$ over $J(\Q_p) \times J(\Q_p)$. By definition, the partial group laws in $\calN$ are just addition keeping one coordinate fixed. Let $[D],[D_1],[D_2]\in J(\Q_p)$ and $[E],[E_1],[E_2] \in J^0(\Q_p)$ and $v_1,v_2 \in \Q_p$. The first group law is
\begin{align*}
&([D_1], [E], v_1) +_1 ([D_2], [E],v_2) =([D_1]+ [D_2], [E], v_1 + v_2 ).
\end{align*}
The second group law is
\begin{align*}
&([D], [E_1], v_1) +_2 ([D], [E_2],v_2) =([D], [E_1]+[E_2], v_1 + v_2 ).
\end{align*}
\begin{definition}
\label{def:Psi}
We define the morphism of biextensions
\begin{align*}
\Psi\colon  \calM^\times(\Z_p) \to \calN
\end{align*}
to be the projection $\calM^\times(\Z_p) \to J(\Q_p) \times J(\Q_p)$ on the first two factors and $\psi$ on the last factor. 
\end{definition}

\begin{remark}
 Since $\log (-1) = 0$,  the morphism $\Psi$ sends the two integer points of $\calM^\times(\Z)$ above a fixed integer point of $J \times J^0$ to the same point. 
\end{remark}

The following proposition appears in \cite[Section~9.3]{EdixhovenLido} but is not proven.
\begin{proposition}
\label{prop:calN}
The map $\Psi\colon  \calM^\times(\Z_p) \to \calN$ is a morphism of biextensions.
\end{proposition}
\begin{proof}
First we show that $\Psi$ is well defined. For divisor classes $[D] \in J(\Q_p)$ and $[E] \in J^0(\Q_p)$ we can always choose representing divisors $D$ and $E$ with disjoint support over $\Q_p$; we show that the choice of representing divisors $D$ and $E$ does not matter.
 Suppose $D  = D' + \Div g$ for some rational function $g$ with $\div g$ disjoint from $E$. Multiplication by $g$ induces an isomorphism $\calO_X(D) \to \calO_X(D')$ sending $E^* 1 \mapsto E^* g(E)$ by Remark~\ref{rem:divGgen}. Under $\psi$, the section $E^* \lambda$ in $E^* \calO_X(D)$ maps to $\log \lambda + h_p(D, E)$ while $E^* g(E) \lambda$ in $E^* \calO_X(D')$ maps to $\log \lambda + \log g(E) + h_p (D', E)$. But since $h_p(\Div g, E) = \log g(E)$ we have the equality $h_p(D', E)  + \log g(E) = h_p (D, E)$, so the choice of representing divisor for $[D]$ does not change the value of $\Psi$. By symmetry of the norm \cite[Section~6.5]{EdixhovenLido}, we can also conclude that $\Psi$ does not depend on the choice of representing divisor for $[E]$.

Finally we show that $\Psi$ preserves the two group laws \eqref{eqn:tensor2onM} and \eqref{eqn:tensor1onM}. Let $[D_1],[D_2]\in J(\Q_p)$, and $[E]  \in J^0(\Q_p)$ with $E$ disjoint from $D_1$ and $D_2$. Let $E^* \lambda_1 \in E^*\calO_X(D_1)$ and $E^*\lambda_2 \in E^*\calO_X(D_2)$. Under $\psi$, the section $E^*\lambda_i$ maps to $\log \lambda_i + h_p(D_i, E)$ for $i = 1, 2$. The group law $\otimes_1$ in $\calM^\times$ sends the sections to $E^*(\lambda_1 \lambda_2)$  in $E^* \calO_X(D_1 + D_2)$. Under the map $\psi$, the section $E^*(\lambda_1 \lambda_2)$ is sent to 
\[\log (\lambda_1 \lambda_2) + h_p (D_1+ D_2, E) =   \log \lambda_1  + \log \lambda_2 + h_p (D_1, E) + h_p (D_2, E).\] Therefore $\Psi$ preserves $\otimes_1$.
By symmetry of the norm it also preserves $\otimes_2$.
\end{proof}

The following proposition relates this to the global $p$-adic height.
\begin{proposition}
\label{prop:intpointscalN} 
Let $[D] \in J(\Z)$ and $[E] \in J^0(\Z)$ with representing divisors $D$ and $E$ that have disjoint support over $\Z_{(p)}$. Let $F$ be the unique vertical divisor such that $F + E$ has multidegree $0$ on all fibers $X_{\F_q}$. Let $z \in \calM^\times([D],[E + F])(\Z)$. Then $\psi(z) = h([D],[E])$ where $h(\cdot,\cdot)$ denotes the global $p$-adic height.
\end{proposition}
\begin{proof}
Let $\calL = \calO_X(D)$. Write $F =\sum_q F_{\F_q}$ where $q$ ranges over the primes of bad reduction for $X$ and $F_{\F_q}$ has support in $X_{\F_q}$. Then by \cite[Proposition~6.9.3]{EdixhovenLido} we have the equation\[
 \calM^\times([D],[E]) = \prod_q q^{-F_{\F_q} \cdot D}\Norm_E(\calL^\times)
\]
where $q$ ranges over the bad primes.

Recall that $\Norm_E(\calL^\times)$ is by definition $\Norm_{E/\Spec \Z}(\calL^\times|_E)$; this torsor is canonically identified with $\calO_{\Spec \Z}(\prod_q q^{-(E \cdot D)_q})^\times$ and hence has generator $\prod_q q^{-(E \cdot D)_q}$, where $(E \cdot D)_q$ denotes the intersection number of $E$ and $D$ over $\Z_{(q)}$ taking values in $\Z$.

In total, we see that under these identifications $\calM^\times([D],[E + F])$ is generated by the element $E^* \prod_q q^{-((E + F) \cdot D)_q}$. By definition, for $q \neq p$, we have that $h_q(D,E)$ is $-((E + F) \cdot D)_q \log q$, and hence we get
\begin{align*}
\psi(z) &= \log \prod_q q^{-((E + F) \cdot D)_q} + h_p(D,E)\\
    &= \sum_{q \neq p} h_q (D,E) + h_p(D,E)\\
    &= h([D],[E])
\end{align*}
as we wanted.
\end{proof}

\subsection{The torsor \texorpdfstring{$T_f$}{Tf}}
\label{subsec:Tf}
We set up some notation. Recall from Section~\ref{sec:overview} that we have fixed a simple open set $U \subset X^\sm$ that contains the smooth points of one geometrically irreducible component of each fiber.
Let $f$ be a trace zero endomorphism of $J$. Recall the integer $m$ from \eqref{eq:defm}. The map $m\cdot \circ f$ is a morphism $J \to J^0$. Let $c \in J(\Z)$ denote the unique element such that $j_b^* (\id,m\cdot \circ\tr_{c} \circ f)^* \calM^\times$ is trivial over $U$. Let $\alpha_f \colonequals m\cdot \circ \tr_c \circ f$. Let $\xi_f\colon  T_f \to J$ denote the $\G_m$-torsor $(\id,\alpha_f)^* \calM^\times$ over $J$. The trivialization of $j_b^* (\id,m\cdot \circ\tr_{c} \circ f)^* \calM^\times$ then gives us a morphism $\widetilde{j_{b,f}}\colon  U \to T_f$ of schemes over $J$.
\begin{remark}
If $f$ is identically zero, then $T_f$ is isomorphic to the trivial $\G_m$-torsor over $J$. If $r<g$ this reduces to the geometric linear Chabauty case, see \cite{PimThesis,geometricLinearChab}
for more details, but when $r  = g$ this trivial torsor contains no information.
\end{remark}

As discussed in the overview, we work on the curve residue disk by residue disk, and hence we will describe the residue disks of $T_f$, culminating in Lemma~$\ref{lem:paramTf}$. Throughout the rest of this section, fix a $\overline{t} \in T_f(\F_p)$. We work inside the residue disk $T_f(\Z_p)_{\overline{t}}$. Since $T_f$ is trivial on fibers, the residue disk $T_f(\Z_p)_{\overline{t}}$ is isomorphic to $J(\Z_p)_{\xi_f(\overline{t})} \times \G_m(\Z_p)_u$ for some unit $u \in \F_p$. We would like to parametrize this residue disk.

\begin{definition}
\label{def:param}
Let $Y$ be a smooth scheme over $\Z_p$ of relative dimension $d$, and let $y \in Y(\F_p)$. We say $t_1,\dots,t_d$ are \defi{parameters} of $Y$ at $y$ if they are elements of the local ring $\calO_{Y,y}$ such that the maximal ideal is given by $(p,t_1,\dots,t_d)$. 

Define $t_i' \colonequals t_i/p$.
Then evaluation of $t'$, the vector $(t_1',\dots,t_d')$, gives a bijection $t'\colon Y(\Z_p)_y \to \Z_p^d$.
We call $t'$ a \defi{parametrization} given by parameters $t_i$.
\end{definition}
\begin{example}
\label{ex:gmparam}
Take $Y = \G_m = \Spec \Z_p[x,x^{-1}]$ over $\Z_p$; this is of relative dimension $1$. Let $y = 1 \in \G_m(\F_p)$. Then $x-1$ is a parameter at $y$; it induces a parametrization $\theta\colon  \G_m(\Z_p)_{y} \to \Z_p$ given by $u \mapsto (u-1)/p$. Note that the map $\log$, defined by its power series $\log (1+x) = x - \frac{x^2}{2} + \cdots$ also induces a bijection $\phi = \log/p\colon  \G_m(\Z_p)_{y} \to \Z_p$, but this is \emph{not} a parametrization; it is not given by evaluating elements of the maximal ideal, and is not even fully algebraic in nature. However, there is a relation between $\phi$ and $\theta$, in that $\theta \circ \phi^{-1}$ is given by the power series $\frac{1}{p} \big(xp - \frac{(xp)^2}{2} + \cdots\big) \in \Z_p[[x]]$.
\end{example}

In \cite[Lemma~6.6.8]{EdixhovenLido} the residue disk $T_f(\Z_p)_{\overline{t}}$ is parametrized using parameters at $\overline{t}$. However, this parametrization can be difficult to work with because it uses parameters in $J$. The group law of $J$ expressed in these parameters is given by complicated converging power series. It is possible to use this parametrization in practice: see for example \cite{mascot18}, where the Khuri-Makdisi representation \cite{makdisi04} is generalized in order to work with points of the Jacobian up to the required $p$-adic precision and compute parameters of them; however, with this representation other steps of the algorithm, like computing the image under an endomorphism, would be more difficult. Here, we opt to use the logarithm of $J$ instead to give a bijection between the residue disk $T_f(\Z_p)_{\overline{t}}$ and $\Z_p^{g+1}$ that is not a parametrization in the sense of Definition~\ref{def:param}. For a definition of this logarithm, see \cite{HondaCommutativeFormalGroups}. To describe the relationship between this bijection and the parametrization of this residue disk we need the framework of convergent power series.
\begin{definition}
Let $n \in \N$. The \defi{ring of convergent power series in $n$ variables} is defined as 
\begin{align*}
\Q_p\langle x_1,\dots,x_n \rangle \colonequals \{\sum_{I \in \N^n} a_I \x^I \in \Q_p[[x_1,\dots,x_n]] \mid \lim_{I \to \infty} |a_I| = 0 \}
\end{align*}
where $\x = (x_1,\dots,x_n)$ is the vector of variables.
An element of this ring is called an \defi{integral convergent power series} if it lies inside $\Z_p[[x_1,\dots,x_n]]$.
The convergent power series are those power series converging on all of $\Z_p^n$. Unlike formal power series, one can always compose two (integral) convergent power series, since by definition the resulting infinite sum inside the ring of (integral) convergent power series converges.
\end{definition}

\begin{remark}
\label{rem:paramparam}
Let $Y$ be a smooth scheme over $\Z_p$ of relative dimension $d$, let $y \in Y(\F_p)$, and let $\theta,\theta'\colon  Y(\Z_p)_y \to \Z_p^d$ be two parametrizations. Then the composite $\theta' \circ \theta^{-1}\colon  \Z_p^d \to \Z_p^d$ is given by (multivariate) integral convergent power series that are linear modulo $p$, and in fact are of degree at most $M$ modulo $p^M$. 
\end{remark}

\begin{lemma}
\label{lem:our37}
Let $G$ be a smooth, commutative group scheme over $\Z_p$ of relative dimension $d$. Let $G(\Z_p)_0$ be the residue disk containing the unit $0 \in G(\Z_p)$. Let $\theta\colon  G(\Z_p)_0 \to \Z_p^d$ be a parametrization, and let $\log\colon  G(\Z_p)_0 \to p\Z_p^d$ be a choice of logarithm. Then $\log \circ \theta^{-1}\colon  \Z_p^d \to p\Z_p^d$ is given by $d$ integral convergent power series in $d$ variables. For $n \geq 0$ the coefficient of a degree $n$ monomial in one of these power series has valuation at least $\max(1,n - v_p(n))$. 
\end{lemma}
\begin{proof}
By \cite[Lemma~3.7]{PimThesis} the function $\log \circ \theta^{-1}$ is given by integral convergent power series. There the third author gives the vector-valued formula
\[
\log = \sum_{I \in \N^d \setminus (0,\dots,0)} a_I c_{|I|} \x^I
\]
where $\x = (x_1,\dots,x_d)$ is the vector of variables, the coefficients $a_I$ lie in $\Z_p$, the notation $|I|$ means $i_1 + \cdots + i_d$ where $I = (i_1,\dots,i_d)$, and $c_n = p^n/n$. (In this paper we do not divide by $p$ in the log, unlike in \cite{PimThesis}). The result follows immediately from the observation that $v_p(c_{|I|}) = |I| - v_p(|I|)$.  
\end{proof}

The following result establishes the analyticity of the map $\psi$ on residue disks of $\calM^\times$.
\begin{lemma}[{\cite[Section~9.3]{EdixhovenLido}}]
\label{lemma:heightanalytic}
Let $\ol{z} \in \calM^\times(\F_p)$. 
Let $\widetilde{z}$  
be a lift of $\ol{z}$ to $\calM^\times(\Z_p)$. Let $\Theta\colon  \Z_p^{2g+1} \to \calM^\times(\Z_p)_{\ol{z}}$ be a parametrization.
Consider the map 
\begin{align*}
    \psi_{\ol{z}}\colon 
    \calM^\times(\Z_p)_{\ol{z}} &\to \Q_p \\
    z &\mapsto \frac{\psi(z) - \psi(\widetilde{z})}{p}.
\end{align*}
Then $\psi_{\overline{z}} 
\circ \Theta$ is given by a convergent power series. 
\end{lemma}

As discussed above, we can now find a bijection between residue disks of $T_f$ and $\Z_p^{g} \times \Q_p$. We use the logarithm of the Jacobian, which gives an isomorphism $\log\colon  J(\Z_p)_0 \to p\Z_p^g$  by choosing a basis of $H^0(J_{\Z_p}, \Omega^1)$ as well as the map $\psi$ defined in \eqref{def:calN}. For ease of notation, we suppress the monomorphism $T_f \to \calM^\times$ in our notation, and apply $\psi$ directly to $T_f(\Z_p)$. 
\begin{definition}
\label{def:paramTf}
Recall that we fixed a $\overline{t} \in T_f(\F_p)$.
Choose $\tilde{t} \in T_f(\Z_p)_{\overline{t}}$ to be a lift of $\overline{t}$. Let $\phi_f\colon  T_f(\Z_p)_{\overline{t}} \to \Z_p^{g} \times \Q_p$ be defined by\[\phi_f(z) = ((\log \xi_f(z) - \log \xi_f(\tilde{t}))/p, (\psi(z)- \psi(\tilde{t}))/p)\]
where $\psi$ is defined in \eqref{def:calN} and the map $\xi_f\colon  T_f \to J$ is the structure morphism of $T_f$.

We call $\phi_f$ a \defi{pseudoparametrization} of the residue disk $T_f(\Z_p)_{\overline{t}}$.
\end{definition}
Similarly to Example~\ref{ex:gmparam}, this is not a parametrization; it shares some of the properties of a parametrization, notably the property in Remark~\ref{rem:paramparam}, as the following lemma shows.
\begin{lemma}
\label{lem:paramTf}
The pseudoparametrization $\phi_f$ is an injection, and for any parametrization $\theta\colon  T_f(\Z_p)_{\overline{t}} \to \Z_p^{g+1}$ the resulting map $\phi_f \circ \theta^{-1}\colon  \Z_p^{g+  1}\to \Z_p^{g} \times \Q_p$ is given by $g + 1$ convergent power series. The valuation of the coefficient of any degree $n$ monomial occurring in one of the first $g$ convergent power series is at least $\max(0,n-1-v_p(n))$. 
\end{lemma}
\begin{proof}
By Lemma~\ref{lem:our37} and Lemma~\ref{lemma:heightanalytic} the pseudoparametrization is given by convergent power series and the valuations of the coefficients behave in the required way. It remains to prove that it is an injection. First, note that the maps $\frac1p\log\colon J(\Z_p)_0 \to \Z_p^g$ and $\frac1p\log \colon \G_m(\Z_p)_1 \to \Z_p$ are bijections.

Let $[D], m(f([D])+c) \in J(\Z_p)_0$ with  disjoint representing divisors $D$ and $E$, and let $\lambda_0,\lambda_1 \in \G_m(\Z_p)$ such that for $i = 0,1$ we have $([D],[E],\lambda_{i}) \in T_f(\Z_p)_{\ol{t}}$. Assume that $\phi_f(([D],[E],\lambda_{0})) = \phi_f(([D],[E],\lambda_{1}))$. Then we have that $\log\lambda + h_p(D,E) = \log \lambda' + h_p(D,E)$ so, because $\frac1p\log$ is injective on residue disks, then $\lambda = \lambda'$, and $\phi_f$ is injective. 

By Lemma~\ref{lem:our37} the result follows.
\end{proof}

\subsection{The torsor \texorpdfstring{$T$}{T}}
\label{subsec:T}

Let $f_1, \dots, f_{\rho -1}$ be a basis for the trace zero endomorphisms of $J$. We simplify our notation by setting $c_i \coloneqq c_{f_i}$, $\alpha_i \coloneqq \alpha_{f_i}$, $T_i \coloneqq T_{f_i}$, and $\xi_i \coloneqq \xi_{f_i}\colon  T_{i} \to J$. 

Now we define $\xi\colon  T \to J$ to be the $\G_m^{\rho-1}$-torsor given by the fiber product 
\[T \colonequals T_1 \times_J T_2 \times_J \cdots \times_J T_{\rho - 1}.\]
Finally, let $\widetilde{j_b} \colon  U \to T$ be a choice of morphism (well defined up to the choice of $\rho -1$ signs) coming from the morphisms $\widetilde{j_{b,f_i}}\colon  U \to T_i$.

As in Section~\ref{subsec:Tf}, we can pseudoparametrize residue disks of $T$.

\begin{definition}
\label{def:paramT}
Recall that we fixed a $\overline{t} \in T(\F_p)$.
We also fix $\tilde{t} = (\tilde{t}_1,\dots,\tilde{t}_{\rho-1})\in T(\Z_p)_{\overline{t}}$ a lift of $\overline{t}$. Let $\phi\colon  T(\Z_p)_{\overline{t}} \to \Z_p^{g} \times \Q_p^{\rho-1}$ be defined by\[\phi(z_1,\dots,z_{\rho-1}) = ((\log \xi_1(z_1) - \log \xi_1(\tilde{t}_1))/p, (\psi(z_1)-\psi(\tilde{t}_1)/p),\dots,(\psi(z_{\rho-1})-\psi(\tilde{t}_{\rho-1})/p))\]
where $\psi$ is defined in \eqref{def:calN}.
We call $\phi$ a \defi{pseudoparametrization} of the residue disk $T(\Z_p)_{\overline{t}}$. (Recall that $\xi_i(z_i)$ and $\xi_i(\tilde{t}_i)$ are independent of $i$, since $T$ is a fibered product over $J$.)
\end{definition}
\begin{corollary}
\label{lem:paramT}
The pseudoparametrization map $\phi$ is an injection, and for any parametrization $\theta\colon  T(\Z_p)_{\overline{t}} \to \Z_p^{g + \rho -1}$ the resulting map $\phi \circ \theta^{-1}\colon  \Z_p^{g+ \rho - 1}\to \Z_p^{g} \times \Q_p^{\rho - 1}$ is given by $g + \rho-1$ convergent power series. For any of the first $g$ power series, the valuation of the coefficient of a degree $n$ monomial is at least $n-1-v_p(n)$.
\end{corollary}
\begin{proof}
    This is a corollary of Lemma~\ref{lem:paramTf}.
\end{proof}

The main advantage of this method is that for $\phi_f$ we need only to compute the map $\psi$ defined in \eqref{def:calN}; it is this fact that allows to us to mainly work in $\calN$ and only translate back to the image of the residue disk under $\phi$ when needed.

\section{The line bundle}
\label{sec:linebundle}

In this section we describe how to explicitly construct the nontrivial $\G_m$-torsor $T$ and give a formula for the section $\widetilde{j_b}\colon  U \to T$. For this, we work with endomorphisms of $J$. We make this explicit by considering correspondences on $X_\Q \times X_\Q$ and extensions on $U \times X$. Recall that $p>2$ is henceforth a prime of good reduction.

\begin{remark}
\label{rem:regmodelproj}
To work with divisors on $U,\,X$ or $U \times X$ explicitly, we use equations for a projective regular model of $X$. There are multiple ways to do this. On a theoretical level, a regular model itself is projective over $\Z$ because it is a repeated blowup of the projective closure of its generic fiber. On a practical level, this process could embed the regular model in a high-dimensional projective space, and it is easier to work on affine patches. In this case we give divisors on each of the affine patches by Gr\"obner bases, compatible with the glueing data. For a practical implementation, we recommend this latter method. This is implemented in \texttt{Magma}, for example. The methods in the rest of the section are agnostic to the exact implementation. Throughout this section, we assume we can represent effective divisors on the regular model by a Gr\"obner basis, and we represent general divisors by a difference between two effective divisors.
\end{remark}

As explained in Section~\ref{subsec:T}, to construct the torsor $T$, we need $\rho - 1$ independent trace zero endomorphisms $(f_i)_{i=1}^{\rho -1}\colon  J \to J$. (In general one only needs $n$ independent nontrivial trace zero endomorphisms where $n$ is such that $r < g + n$, but one expects to obtain a smaller superset of $p$-adic points containing $X(\Z)$ for higher $n$. In fact, if we use $n$ nontrivial independent endomorphisms such that $r<g+n-1$, then we expect to cut out $X(\Z)$ exactly unless there is some geometric reason for extra points.)
To work with any endomorphism $f\colon  J \to J$ explicitly, we recall some facts about correspondences, as can be found in \cite{BSmith}. A \defi{correspondence} on $X \times X$ is a divisor $D$ on $X \times X$.

Write $D = \sum_i n_i D_i$ as a sum of prime divisors. Denote by $\pi_1^{D_i}\colon  D_i\to X$ the projection onto the first factor of $X \times X$ and similarly $\pi_2^{D_i}$ for projection onto the second factor. 
The correspondence $D$ induces an endomorphism of the Jacobian $\xi_D = \sum_i n_i \pi_{2,*}^{D_i }\pi_1^{D_i ,*}$. In particular, it sends the Jacobian point $[x-y]$ to $\calO_X(D|_{x \times X} - D|_{y \times X})$.
\begin{example}
Consider negation $-1\cdot \colon  J \to J$ on a hyperelliptic curve of the form $y^2 = h(x,z)$ in weighted projective space. If we give $X \times X$ the projective coordinates $x,y,z,x',y',z'$, then a correspondence representing $-1\cdot$ is given by the homogeneous equation $y = -y'$.
\end{example}

The aim of this section is to describe, given correspondences for all $f_i$, how to calculate the morphism $\widetilde{j_b}\colon  U \to T$. For this goal, we partially follow \cite[Section~7]{EdixhovenLido}.

In the case where $X_\Q$ is a classical modular curve we can construct many trace zero endomorphisms using the Hecke algebra. See for example the computation leading to \eqref{eqn:matrixRepf} in Section~\ref{sec:example}.

We now focus on the computations for a single trace zero endomorphism $f\colon J\to J$.
We can compute equations for a correspondence $D_{f,\Q} \subset X_\Q \times X_\Q$ inducing $f$ using the code of Costa, Mascot, Sijsling, and Voight \cite{RigorousEndo}.
The input of that algorithm is the $g\times g$ matrix giving the representation of the morphism $f$ on a basis of differential forms $H^0(X_{\Q}, \Omega^1)$.

\begin{myalgorithm}[Compute $A_\alpha$] \label{alg:Aalpha}
\hfill\\
Input: $D_{f,\Q} \subset X_\Q \times X_\Q$ a divisor.
\\
Output: a divisor $A_\alpha$ on $X^\sm \times X$.
\begin{enumalg}
\item Spread out $D_{f,\Q}$ to $D_f'$ over $X^\sm \times X$ by clearing denominators of the generators of the Gr\"obner basis.
\item Set $B \colonequals D_f'|_{X^\sm \times b}$ and $C \colonequals D_f'|_{\Delta_{X^\sm}}$.
\item Set $A_\alpha \colonequals m\left(D_f' - B \times X + X^\sm \times B - X^\sm \times C\right)$ (where $m$ is defined in \eqref{eq:defm}).
\item Return $A_\alpha$, as a Gr\"obner basis over $\Z$.
\end{enumalg}
\end{myalgorithm}

\begin{lemma}\label{lemma:A-alpha}
    The divisor $A_\alpha$ on $X^{\sm} \times X$ given by Algorithm~\ref{alg:Aalpha} is the unique divisor on $X^\sm \times X$ with the following properties:
    \begin{enumerate}[(a)]
      \item \label{condition:endo} the endomorphism of $J$ induced by the correspondence $A_\alpha$ is $m\cdot \circ f$;
      \item \label{condition:secondproj} $\calO_{X^\sm}(A_\alpha|_{U \times b})$ is rigidified with trivializing section $1$;
      \item \label{condition:diagcond} $\calO_{X^\sm}(A_\alpha|_{\Delta})$ is rigidified, compatible with the previous rigidification;
      \item \label{condition:firstproj} the degree of $A_\alpha$ restricted to fibers of the first projection is $0$.
    \end{enumerate}
\end{lemma}

\begin{proof}
By \cite[Theorem~3.4.7]{BSmith}, any divisor inducing the endomorphism $m\cdot \circ f$ is of the form $mD_f + F$ such that $F$ is a sum of vertical or horizontal divisors, so then \ref{condition:endo} holds.
Conditions \ref{condition:secondproj} and \ref{condition:diagcond} force $F$ to be $m( - B \times X + X^\sm \times B - X^\sm \times C)$.
Finally, by \cite[Proposition~11.5.2]{birkenhake} and the important fact that the trace of $f$ is zero we have that $\deg(A_\alpha|_{P \times X}) = 0$ and \ref{condition:firstproj} holds. So $ A_\alpha$ is the desired divisor.
\end{proof}
\begin{remark}
Conditions \ref{condition:secondproj} and \ref{condition:firstproj} are the other way from the order chosen in Edixhoven--Lido, in order to agree with the convention in \cite{RigorousEndo}. (That is, in Edixhoven--Lido, they require that the fibers of the \emph{second} projection are degree $0$.)
\end{remark}

This divisor $A_\alpha$ determines a line bundle $\calL_\alpha = \calO_{X^\sm \times X} (A_\alpha)$ on $X^{\sm}\times X$, rigidified on $X^\sm\times b$, of degree $0$ on the fibers of the first projection, and such that $\Delta^* \calL_\alpha$ is trivial. This induces the endomorphism $m\cdot \circ f$ by
\begin{align}
\label{eqn:alphaPminQ}
[x-y] \mapsto (\Lc_{\alpha})_{x \times X} \otimes (\Lc_{\alpha})_{y \times X}^{-1}. \end{align}

\begin{corollary}
\label{eqn:cDefinition}
Let 
$c \colonequals \left[(\Lc_{\alpha,\Q})_{b \times X}\right] \in J(\Q) = J(\Z). $
Let $\alpha = m\cdot \circ \tr_c \circ f$ be the morphism $\alpha\colon J\to J^0$. Then $j_b^* (\textnormal{id},\alpha)^* \calM^\times$ is trivial over $U$.
\end{corollary}
\begin{proof}
This follows directly from \cite[Proposition~7.2]{EdixhovenLido}
\end{proof}

The rest of this section will be dedicated to computing $\alpha$, and computing the trivialization of $j_b^* (\id,\alpha)^* \calM^\times$. 
\begin{myalgorithm}[Compute $c$]\hfill
\label{alg:computec}

\noindent Input: equations for a correspondence $A_\alpha$ output by Algorithm~\ref{alg:Aalpha}, inducing the morphism $m\cdot \circ  f\colon  J \to J$. 
\\
Output: a divisor representing $c = \left[(\Lc_{\alpha})_{b \times X}\right] \in J(\Q) = J(\Z)$.

\begin{enumalg}

\item Set $A_f \colonequals A_{\alpha}/m$ (recall that $A_{\alpha}$ was defined as $m$ times a different correspondence, so this is well defined).
\item Compute the generic fiber $A_{f,\Q}$ of $A_f$.
\item Compute equations for the divisor $A_{f,\Q}|_{b \times X}$ by specializing the equations of $A_{f,\Q}$ to $b$ in the first copy of $X^{\sm}$.
\item Return a Gr\"obner basis for $A_{f,\Q}|_{b \times X}$ over $\Q$.
  \end{enumalg}
\end{myalgorithm}

\begin{myalgorithm}[Compute $f_*$]\hfill

\label{alg:push}
\noindent Input: a morphism of projective schemes $f\colon  X \to Y$ given as a graded ring morphism $f^*\colon  S \to R$, where $X = \Proj R$ and $Y = \Proj S$; an irreducible subvariety $Z$ of $X$ given by a Gr\"obner basis for its defining ideal $J$ in $R$.\\
\noindent Output: the pushforward $f_*([Z])$, given by a Gr\"obner basis.

\begin{enumalg}
  \item Let $B$ be a set of generators of $S$.
  \item Set $I \subset S \tensor R$ to be the ideal generated by $\{b\tensor 1 - 1 \tensor f^*(b) \mid b \in B\}$ and $1 \tensor J$.
  \item Compute a Gr\"obner basis $B$ for $I$ with respect to the lexicographical ordering on $S \tensor R$.
  \item Set $K \colonequals I \cap S$ with Gr\"obner basis $B \cap S$.
  \item \label{alg:push:step4} Compute the degree $d \coloneqq \deg \left(f|_Z\colon  \Proj R/J \to \Proj S/K\right)$.
  \item Return a Gr\"obner basis for $K^d$.
\end{enumalg}

\end{myalgorithm}
\begin{proof}
By construction, $K$ is the defining ideal for the image of $Z$. The pushforward of $Z$ is then exactly $(\deg f|_Z)\cdot[\im f|_Z]$.
\end{proof}

\begin{remark}
In Step~\ref{alg:push:step4}, we need to compute the degree of a morphism between projective schemes. There are algorithms to compute the degree of a rational map between two projective schemes. See for example \cite{stagliano} for a discussion on an implementation in Macaulay2.
\end{remark}

\begin{myalgorithm}[Apply $f$]\hfill
\label{alg:applyf}

\noindent Input: a ring $S$ and two effective divisors $D_+$ and $D_-$ on $X_S^\sm$ of the same degree; the correspondence $A_\alpha$ from Algorithm~\ref{alg:Aalpha} inducing the morphism $m\cdot f\colon  J \to J$.

\noindent Output: the Jacobian point $m\cdot f([D_+ - D_-]) \in J(S)$.
  
  \begin{enumalg}
\item For $D \in \{D_+,D_-\}$ do:
    \begin{enumalg}
      \item Compute a Gr\"obner basis for $A_\alpha|_{D \times X}$ as a divisor on $D \times X$.
      \item Write $D = \sum_i  n_i D_i$ as a sum of irreducible components using primary decomposition.
      \item Compute the Gr\"obner basis for the pushforward $E(D_i)\colonequals n_i f_*(D_i)$ on $X$ using Algorithm~\ref{alg:push} for every $D_i$.
      \item Set $E(D) \colonequals \sum_i E(D_i)$.
    \end{enumalg}
\item Return $E(D_+) - E(D_-)$.
\end{enumalg}
\end{myalgorithm}

\begin{remark}
In the case where one can write $[D_+ - D_-]$ as a sum $\left[\sum_{i=1}^k n_i P_i\right]$ of $S$-points, one can use the isomorphism $P_i \times X \simeq X$ to simply compute $A_\alpha|_{P_i \times X}$ on $X$ and take the linear combination $\left[\sum_{i=1}^k n_i A_\alpha|_{P_i \times X}\right]$.
\end{remark}
  
Finally, we discuss the section $\widetilde{j_b}\colon  U \to T$ lying above the Abel--Jacobi map $j_b\colon  U \to J$ with base point $b$. Let $\bar{z} \in X(\F_p)$.
Since the pullback $j_b^* T$ is trivial, there is a morphism $\widetilde{j_b}\colon  U \to T$ embedding each residue disk $U(\Z_p)_{\overline{z}}$ into the $(g+\rho - 1)$-dimensional residue disk $ T(\Z_p)_{\widetilde{j_b}(\overline{z})}$. To compute this map, we follow \cite[ Section~7]{EdixhovenLido}. Let $n$ be the product of all primes of bad reduction. We first need to compute the numbers $W_q$ and $V_q$ mentioned in \cite[Proposition~7.8]{EdixhovenLido} for $q \mid n$. These numbers have an involved definition in general. Nevertheless, they can be explicitly computed in our case, and we explain their meaning below.

By Lemma~\ref{lemma:A-alpha} the line bundles $\Delta^*(\Lc_{\alpha})$ and $(\id, b)^*(\Lc_{\alpha})$ are trivial with trivializing sections $\ell = 1$. Then $W_q$ is defined as the valuation of this section $\ell$ on $U_{\F_q}$. In our case, these are always $0$. It remains to compute $V_q$. We recall the definition. Note that $\Lc_\alpha$ has degree $0$ on the fibers of the projection $U \times X \to U$, but it might not have multidegree $0$.

\begin{definition}
\label{def:Vq}
We define $V$ to be the unique vertical divisor on $U \times X$ having support disjoint from $U \times b$ such that $\Lc_\alpha(V)$ has multidegree $0$ on all fibers of the projection. Write $V_{\F_q}$ as a sum of irreducible components of $U_{\F_q} \times X_{\F_q}$, i.e., as a linear combination of $U_{\F_q} \times Y_{\F_q}$ where $Y_{\F_q}$ is an irreducible component of $X_{\F_q}$. For $q\mid n$ define $V_q \in \Z$ to be the coefficient of the component $(U_{\F_q} \times U_{\F_q})$ in $V_{\F_q}$. 
\end{definition}

\begin{lemma}
\label{lem:Vq}
The local height $h_q(z-b, A_{\alpha}|_{z\times X})$ is equal to $-V_q \log q$ for any $z \in U(\Z_q)$.
\end{lemma}
\begin{proof}
Since $V$ is the unique vertical divisor with $A_\alpha + V$ having multidegree $0$ on all fibers of the projection, we have that $h_q(z-b, A_{\alpha}|_{z\times X})$ is equal to $-((z-b) \cdot (A_\alpha + V)|_{z \times X})_q \log q$. By construction, the divisors $z-b$ and $A_{\alpha|_{z \times X}}$ are disjoint over $\Z$, hence it remains to show that $((z-b) \cdot V|_{z \times X})_q = V_q$. This follows from Definition~\ref{def:Vq} and the fact that $V$ has support disjoint from $U\times b$. \end{proof}

To compute these numbers, we give the following algorithm. 

\begin{myalgorithm}[Calculate $V_q$]\hfill
\label{alg:calculateVq}

\noindent Input: the curve $X$, a bad prime $q$ dividing $n$, the open set $U$ such that $U(\F_q) \neq \emptyset$, and the divisor $A_\alpha$ on $X \times X$.\\
\noindent Output: the integer $V_q$.
  
\begin{enumalg}
\item \label{step:smoothpt} Pick a point $\overline{Q}\in U(\F_q)$. 
\item Compute $A_{\alpha}|_{\overline{Q} \times X}$.
\item Compute the multidegree of  $A_{\alpha}|_{\overline{Q} \times X}$.
\item Compute the multidegree of the irreducible components of $X_{\F_q}$.
\item Compute the unique linear combination $D \subset X_{\F_q}$ of these irreducible components such that $D$ does not meet $\overline{b}$ and such that $A_{\alpha}|_{\overline{Q} \times X} + D$ has multidegree $0$ at the fiber over $q$.
\item Set $V_q$ to be the coefficient of the irreducible component containing $U_{\F_q}$ in $D$.
\item Return $V_q$.
\end{enumalg}
\end{myalgorithm} 


\begin{remark}
These local heights can also be computed using harmonic analysis on the dual graph, see \cite[Section~12]{BettsDogra}. Even though both the geometric method and the harmonic method can be realized as combinatorics on the dual graph, it is not clear how to compare the two computations of local heights.
\end{remark}

Let $R$ be a ring and $z \in U(R)$.  By \cite[Proposition~7.5]{EdixhovenLido} we have 
\begin{align*}
T_f(j_b(z)) & =  \calM^\times(j_b(z), \alpha (j_b(z)) ) =z^* (z, \id)^* (\Lc_{\alpha})^\times \otimes b^* (z, \id)^* (\Lc_{\alpha})^{\times,-1} \\
&= (\Lc_{\alpha})^\times (z,z) \otimes (\Lc_{\alpha})^\times (z,b)^{-1} = (\Lc_{\alpha})^\times(z,z).\notag
\end{align*}

We apply \cite[Proposition~7.8]{EdixhovenLido} to give a formula for $\widetilde{j_b}(z)$ when $R \subset \Z_p$. We have that
\begin{align} 
\label{def:jbtilde}
    \widetilde{j_b}(z)   = \prod_{q\mid n} q^{-V_q} (z^* 1) \otimes (b^* 1)^{-1} = (z-b)^* \prod_{q\mid n} q^{-V_q} \in (z-b)^* \calO_X(A_\alpha|_{z \times X})
\end{align}
is a trivializing section over the curve.
The image in $\calN$ is given by
\begin{align}
\label{eq:psijbtilde}
    \psi(\widetilde{j_b}(z)) = h_p(z-b,A_\alpha|_{z \times X}) -\sum_{q\mid n} V_q \log q.
\end{align}
\begin{corollary}
\label{cor:jbtilde}
The function $\Psi \circ \widetilde{j_b}\colon  U(\Z_p) \to \calN$ is given by
\[
z \mapsto ([z-b], [A_\alpha|_{z\times X}], h_p(z-b,A_\alpha|_{z \times X}) -\sum_{q\mid n} V_q \log q).
\]
\end{corollary}

\section{Embedding the curve}
\label{sec:embeddingcurve}

We now describe how to compute the embedding of the curve into the torsor through the evaluation of the trivializing section $\widetilde{j_b}$ on a residue disk of the point $\overline{P} \in U(\F_p)$.
Recall the pseudoparametrization $\phi\colon T(\Z_p)_{\widetilde{j_b}(\overline{P})} \to \Z_p^{g} \times \Q_p^{\rho-1}$ from Definition~\ref{def:paramT}.
Let $\nu$ be a local parameter in the residue disk of the curve above $\overline{P}$. We can parametrize this residue disk by evaluating
\begin{align*}
    \Z_p &\to U(\Z_p)_{\overline{P}},  \hspace{.2in}  \nu \mapsto P_\nu.
\end{align*}
 This is also a parametrization in finite precision, i.e. we have bijections $\Z/p^k\Z \to U(\Z/p^{k+1}\Z)_{\overline{P}}$ for any integer $k \geq 1$.
Define the map $\lambda\colon  \Z_p \to T(\Z_p)_{\widetilde{j_b}(\overline{P})}$ to be the composite of this parametrization $\Z_p\to U(\Z_p)_{\overline{P}}$ and $\widetilde{j_b}$. 
In this section, we show how to apply the following proposition.
\begin{proposition}
\label{prop:embeddingcurve}
The map $\phi \circ \lambda: \Z_p \to \Z_p^{g} \times \Q_p^{\rho-1}$ is given by convergent power series. 

The image $\im(\phi \circ \lambda)$ inside $\im \phi$ is cut out by equations $g_1 =\cdots= g_{g+\rho-2} = 0$ where $g_1,\dots,g_{g+\rho-2} \in \Z_p\langle x_1,\dots,x_{g+\rho-1}\rangle$ are integral convergent power series.     
\end{proposition}
\begin{proof}
This follows from Corollary~\ref{lem:paramT} and \cite[Corollary~2, III.4.5]{BourbakiCA}.
\end{proof}

For actual calculations with the convergent power series $\phi \circ \lambda$, we need to lower bound the valuation of the coefficients.

\begin{proposition}
\label{prop:valuationheight}
Let $\nu$ be a coordinate for $\Z_p$. Consider the $g + \rho - 1$ convergent power series given by $\phi \circ \lambda \colon  \Z_p \to \Z_p^g \times \Q_p^{\rho-1}$. 
For any of the first $g$ convergent power series, the valuation of the coefficient of  $\nu^n$ is at least $n-1-v_p(n)$. For any of the last $\rho-1$ power series, the valuation of the coefficient is at least $n-1- 2\lfloor \log_p n \rfloor+v$, where $v$ is an explicit (possibly negative) constant.
\end{proposition}
\begin{proof}
    The result about the coefficients of the first $g$ power series follows from Corollary~\ref{lem:paramT}. 
    
    Let $i \in \{1,\dots,\rho-1\}$. Then \cite[Lemma~4.5]{examplesandalg} states that the Nekov\'a\v{r} height $h_{i,p}^{\Nek}: X(\Q_p) \to \Q_p$  corresponding to the trace zero endomorphism $f_i$ is analytic on residue disks. Let $c = \min\{0,\min_j d_j(\eta)\}$ for the $d_j(\eta)$ defined in \cite[Section~4]{examplesandalg}. Furthermore they show that the valuation of the coefficient of $\nu^n$ is at least $n -1 - 2 \lfloor \log_p n\rfloor + v'$, where $v' \colonequals \min(\ord_p(\gamma_{\text{Fil}}), c + c_2)$, and $\gamma_{\text{Fil}}$ and $c_2$ are explicit constants defined in \cite[Section~4]{examplesandalg}, depending on $f_i$ among other things. (The valuation of the coefficients of $\nu^n$ stated in \cite[Lemma~4.5]{examplesandalg} differs by $n$ from the value given here, because our coordinates differ from theirs by a factor of $p$.)

    In Section~\ref{sec:comparison} we go more into detail about this Nekov\'a\v{r} height. In particular, in Theorem~\ref{thm:QCInektoCG} together with Proposition~\ref{prop:DbzAalpha} we show that $h_{i,p}^{\Nek}(z)$ and $h_p(z-b,A_{\alpha_i}|_{z \times X})$ differ by a factor of $-m$. It follows from Corollary~\ref{cor:jbtilde} that we can take $v \colonequals v' + v_p(m)$.
\end{proof}

\begin{remark}
In the example of Section~\ref{sec:example}, we calculate that the constant $v$ is $0$ for the residue disk of the curve we consider there. We suspect that this constant can often be taken to be $0$, at least in the cases $p > 2g-1$ and $p \nmid \# J(\F_p)$.
\end{remark}

We first present a general algorithm to compute the trivializing section $\phi \circ \lambda$. 
For example, if $p >3$ and $v = 0$, to compute $\widetilde{j_b}(P_\nu)$ in $\calN$ modulo $p$, it suffices to compute $\widetilde{j_b}$ on two values, for example $\widetilde{j_b}(P_0) $ and $\widetilde{j_b}(P_1)$. Since the embedding must be linear in $\nu$ on $U(\Z/p^2\Z)_{\overline{P}}$, we can interpolate between these values to determine the map. In general, to compute $\phi \circ \lambda$ to finite precision, it is enough to determine the map on $\Z/p^k\Z$-points for some large enough $k$. We give an algorithm to compute $\widetilde{j_b}(P)$ when $P$ is a $\Z/p^k\Z$-point.

\begin{myalgorithm}[The trivializing section]\hfill\label{alg:jbtilde}
\hfill

\noindent Input: A point $P_\nu \in U(\Z/p^k\Z)_{\overline{P}}$.\\
Output: The value $\phi \circ \lambda(\nu)$ to finite precision.
\begin{enumalg}
\item Calculate the Coleman integral $\log(P_\nu - b)$.
\item Compute $A_{\alpha_i}|_{P_\nu \times X}$ for each $i = 1, \dots, ( \rho - 1)$ using Algorithm~\ref{alg:applyf}. 
\item \label{step:heightscalc} Calculate all $h_p(P_\nu-b,A_{\alpha_i}|_{P_\nu \times X})$.
\item For each $A_{\alpha_i}$, compute $c_{U,i} \colonequals -\sum_{q|n} V_q \log q$ using Algorithm~\ref{alg:calculateVq}, where $n$ is the product of the primes of bad reduction for $X$.
\item Return \[(\phi \circ\lambda)(\nu) = (\log(P_\nu - P_0), h_p(P_\nu-b,A_{\alpha_1}|_{P_\nu \times X}) + c_{U,1},\dots,h_p(P_\nu-b,A_{\alpha_{\rho-1}}|_{P_\nu \times X}) + c_{U,\rho-1}).\]
\end{enumalg}
\end{myalgorithm}

For the rest of this section, we describe a practical algorithm to do Step \eqref{step:heightscalc} of Algorithm~\ref{alg:jbtilde} in the case where $X$ is a hyperelliptic curve of the form $y^2 = H(x)$. For hyperelliptic curves where $H$ has odd degree, there is an algorithm to compute the local Coleman--Gross height at $p$ of two disjoint divisors given as a sum of points \cite[Algorithm~5.7]{BalaBesserIMRN}. Forthcoming work of Gajovi\'{c} extends this algorithm to even degree models. 

For any $i=1,\ldots,(\rho-1)$  since the divisor $A_{\alpha_{i}}|_{P_\nu \times X}$ on $X_{\Q_p}$ may not split as a sum of points, we instead consider multiples of this divisor $n A_{\alpha_{i}}|_{P_\nu \times X}$ for $n \in \N$. We can hope some large enough multiple splits as a sum of points. Therefore, we must explicitly describe arithmetic in the Jacobian.  For hyperelliptic curves, this process can be done via Cantor's algorithm \cite{Cantor}.  The main idea is to use the Mumford representations of divisors.  We use the implementation of Cantor's algorithm done by Sutherland in \cite[Section~3]{SutherlandFastJac}.  The only extra step is to keep track of the function that realizes the linear equivalence with a Mumford representation of the sum. Even though Sutherland works with even degree models for hyperelliptic curves, the algorithms still apply to our odd degree model hyperelliptic curves (see \cite[p.433]{SutherlandFastJac}).

\begin{remark}
In practice, we represent divisors with ideals of polynomial rings.  We can translate from a Gr\"obner basis of an ideal to a Mumford representation in the following way. Let $Y$ be a hyperelliptic curve over a field $k$ given by $y^2 = H(x)$. Let $\pi\colon  Y \to \PP^1$ be the degree two morphism forgetting $y$. Let $D$ be an effective divisor on the affine chart $k[x,y]/(y^2-H(x))$ of $Y$, given by a Gr\"obner basis. We assume that $D$ and $\iota(D)$ are disjoint. Then we can find a Mumford representation for $D$ by simply taking a Gr\"obner basis with respect to the lexicographical ordering $y \leq x$. If $D$ and $\iota D$ are not disjoint, one can explicitly compute an effective divisor $E$ on $\PP^1$ such that $D-\pi^* E$ is disjoint from $\iota(D-\pi^* E)$, and hence find a Mumford representation for $D - \pi^* E$. 
\end{remark}

We can now give a practical algorithm to compute the local heights at $p$ in Step \eqref{step:heightscalc} of Algorithm~\ref{alg:jbtilde}. 
When $X$ is a hyperelliptic curve of the form $y^2 = H(x)$, given $P_\nu \in U(\Z/p^k \Z)$ we can apply Algorithm~\ref{alg:applyf} to obtain $A_{\alpha_i}|_{P_\nu \times X}$ as a divisor on $X_{\Q_p}$.

\begin{myalgorithm}[Local heights for the trivializing section on a hyperelliptic curve]
\label{alg:localheightshyp}
\hfill

\noindent Input: A point $P_\nu \in U(\Z/p^k\Z)_{\overline{P}}$ on a hyperelliptic curve $Y\colon  y^2 = H(x)$ and the Mumford representation of $A_{\alpha_i}|_{P_\nu \times Y}$ as a divisor on $Y$.\\
\noindent Output: The value $h_p(P_\nu -b, A_{\alpha_i}|_{P_\nu \times Y})$ to finite precision.
\begin{enumalg}
\item Set $n \colonequals 1$.
\item \label{step:composereduce} Use Cantor's Algorithm to compute a Mumford representation $(u_n, v_n)$ and a rational function $s_n$ such that $\Div(u_n, v_n) + \Div s_n = nA_{\alpha_i}|_{P_\nu \times Y}$ \cite{Cantor}.
\item Check if $u_n$ factors completely over $\Q_p$ into linear factors.
\item If yes, set $x_j$ to be the roots of $u_n$ for $j = 1, \dots, \deg(u_n)$. If no, increase $n$ by $1$ and go back to Step \eqref{step:composereduce}.
\item  Set $y_j \colonequals v_n(x_j)$.
\item Set $Q_j \colonequals (x_j, y_j) \in Y(\Q_p)$.
\item Compute $h_p(P_\nu - b, \sum_{j=1}^{\deg(u_n)} Q_j - \deg(u_n) \infty)$ using \cite[Algorithm~5.7]{BalaBesserIMRN}.
\item Return  $(1/n)( h_p(P_\nu-b,\sum_{j=1}^{\deg(u_n)} Q_j - \deg(u_n) \infty ) + \log(s_n(P_\nu - b)))$.
\end{enumalg}
\end{myalgorithm}

Algorithm~\ref{alg:localheightshyp} does not always terminate; we cannot guarantee that eventually  $nA_{\alpha_i}|_{P_\nu \times Y}$ splits completely into a sum of points over $\Q_p$. In theory, we can split any divisor as a sum of points over some finite extension of $\Q_p$. However, working with these field extensions of $\Q_p$ is often currently not possible in practice.

\begin{remark}
    Algorithms \ref{alg:jbtilde} and \ref{alg:localheightshyp} take in a point $P_\nu$ of precision $k$, but their output can be of smaller precision. This depends on the precision loss in the computation of the $p$-adic height; see \cite[Section~6.2]{BalaBesserIMRN}. 
\end{remark}

\section{Integer points of the torsor}
\label{sec:integerpoints}

Next we discuss the integer points of the torsor $T$. We give an algorithm to construct a map $\kappa\colon  \Z_p^{r} \to T(\Z_p)_{\widetilde{j_b}(\overline{P})}$ with image exactly $\overline{T(\Z)}_{\widetilde{j_b}(\overline{P})}$. 

In practice, to give an upper bound on $\#U(\Z)_{\overline{P}}$,  we only need to compute the image of the map $\kappa$ in $T(\Z/p^2\Z)_{\widetilde{j_b}(\overline{P})}$, because after composing with the pseudoparametrization $\phi$ from Definition~\ref{def:paramT} the map $\kappa$ is given by convergent power series.  In fact, in this section we will show that by virtue of our choice of pseudoparametrization, they are given by $g$ homogeneous linear polynomials and $\rho -1$ quadratic polynomials.

For now we restrict to a single trace zero endomorphism $f$ and the corresponding torsor $T_f$. By iterating over the linearly independent trace zero endomorphisms $f_1, \dots, f_{\rho - 1}$ we recover $T$ and $\kappa$.

Note that if the residue disk $T_f(\Z)_{\widetilde{j_b}(\overline{P})}$ is empty, then its $p$-adic closure is also empty, and therefore we do not need to consider $\overline{P}$. If the disk is not empty, then we can find $\tilde{t} \in T_f(\Z)_{\widetilde{j_b}(\overline{P})}$ by arithmetic in the Jacobian. It is enough to consider if the corresponding residue disk $J(\Z)_{j_b(\overline{P})}$ is empty. This is an instance of the Mordell--Weil sieve at $p$. 

As an intermediate step, we need to compute points $Q_{ij}$ on $\calN$, the trivial biextension, that are the image under $\Psi$ (defined in Definition~\ref{def:Psi}) of generating sections on certain fibers of $\calM^\times(\Z)$. 

We construct points on $\calN$ that are the image of generating sections of residue disks of $\calM^{\times, \rho-1}(\Z)$ following the method in Example~\ref{example:addpointsM}.

\begin{myalgorithm}[Compute the $Q_{ij}$]
\label{alg:Qij}
\hfill
\noindent Input:  $G_1, \dots, G_{r'}$ a generating set of the Mordell--Weil group of $J$, a trace zero endomorphism $f\colon  J \to J$.\\
Output: Points $Q_{ij}$ on $\calN$ that are the image of the generating section of $\calM^\times(G_i, f(G_j))(\Z)$ and $Q_{i0}$ that are the image of the generating section of $\calM^\times(G_i, c)(\Z)$ for $1 \leq i, j \leq r'$.
\begin{enumalg}
  \item Compute $E_1, \dots, E_{r'}$ representing divisors of $G_1, \dots, G_{r'}$.
  \item For each $G_i$, use Algorithm~\ref{alg:applyf} to compute representing divisors $D_1, \dots, D_{r'}$ of $f(G_i)$.
  \item Use Algorithm~\ref{alg:computec} to compute a divisor $D_0$ whose class is the point $c \in J(\Z)$.
  \item Compute the local height $h_p(E_i, D_j)$  and $h_p(E_i, D_0)$ for $1 \leq i , j \leq r'$.
  \item Using \cite[Section~2]{vonBommelHolmesMuller}, compute the height $h_\ell(E_i, D_j)$ at $\ell \neq p$ and $h_\ell(E_i, D_0)$ at $\ell \neq p$  for $1 \leq i , j \leq r'$.
  \item Return $Q_{ij} \colonequals (G_i, f(G_j), \sum_{\ell \text{ prime} } h_\ell(E_i, D_j)) $  and  $Q_{i0} \colonequals (G_i,c, \sum_{\ell \text{ prime} } h_\ell(E_i, D_0)) $ for $1 \leq i , j \leq r'$.
\end{enumalg}

\end{myalgorithm}

Let $G_1, \dots, G_{r'}$ be a generating set for the full Mordell--Weil group, with $r'\geq r$.
Let $\widetilde{G_i}$  be a basis for the kernel of reduction $J(\Z) \to J(\F_p)$ for $i = 1, \dots, r$. (Note that the reduction map is injective when restricted to the torsion of $J(\Z)$, so the kernel of reduction is a free $\Z$-module of rank $r$.)
Write
\begin{align*}
    \widetilde{G_i} = \sum_{j=1}^{r'} e_{ij} G_j
\end{align*}
for some $e_{ij} \in \Z$.
Let $\widetilde{G_t}$ denote the projection of $\tilde{t} \in T(\Z)_{\widetilde{j_b}(\overline{P})}$ to $J_{j_b(\overline{P})}$. 
Write 
\begin{align*}
   \widetilde{G_{t}} = \sum_{i=1}^{r'} e_{0i} G_i
\end{align*}
for some $e_{0i} \in \Z$.
Using the biextension group laws and the points $Q_{ij}$ we construct a series of points in $\calM^\times(\Z)$ living over certain points in $J \times J$ that are the image of generating sections of the corresponding residue disks in $\calM^\times(\Z)$.

A formula for the points $P_{ij}$ over $(\widetilde{G_i}, f(m\widetilde{G_j}))$ is
\begin{align}
\label{eqn:Pij}
    P_{ij} \colonequals \sideset{}{_1}\sum_{k=1}^{r'} e_{ik} \cdot_1\left( \sideset{}{_2}\sum_{\ell=1}^{r'} m\cdot_2e_{j\ell} \cdot_2  Q_{k \ell} \right).
\end{align}
Here, $\cdot_i$ and $\sum_i$ for $i = 1, 2$ denote the biextension group laws \eqref{eqn:tensor1onM} and \eqref{eqn:tensor2onM}.

Next $R_{i\tilde{t}}$ live over $(\widetilde{G_i},\alpha(\widetilde{G_t}))$ and hence
\begin{align}
\label{eqn:Rit}
    R_{i\tilde{t}} \colonequals \sideset{}{_1}\sum_{k=1}^{r'} e_{ik} \cdot_1 \left( m\cdot_2Q_{k 0} +_2  \sideset{}{_2}\sum_{\ell=1}^{r'} m\cdot_2e_{0 \ell}\cdot_2 Q_{k\ell}   \right).
\end{align}

Finally, $S_{\tilde{t} j}$ live over $(\widetilde{G_t},  f(m\widetilde{G_j})) $ and so 
\begin{align}
\label{eqn:Stj}
    S_{\tilde{t} j} \colonequals \sideset{}{_1}\sum_{k=1}^{r'} e_{0k} \cdot_1 \left( \sideset{}{_2}\sum_{\ell=1}^{r'} m\cdot_2e_{j \ell}\cdot_2 Q_{k\ell}   \right).
\end{align}

\begin{remark}
In $\calM^\times(\Z)$, these points are all unique up to sign. Since we are recording the image in $\calN$, this sign does not matter.
\end{remark}

For $n=(n_1, \dots,  n_r) \in \Z^r$ we can now construct the points  $A_{\tilde{t}}(n)$, $B_{\tilde{t}}(n)$, $C(n)$, and $D_{\tilde{t}}(n)$ in $T(\Z)$ given by \cite[(4.2)-(4.4)]{EdixhovenLido}. The key property of this construction is that $D_{\tilde{t}}(n)$ lies above the point $\widetilde{G_t} + \sum_i n_i \widetilde{G_i} \in J(\Z)_{j_b(\ol{P})}$. Furthermore, by \cite[(4.6)-(4.9)]{EdixhovenLido}, we have that $D_{\tilde{t}}((p-1)n)$ is in the residue disk $T_f(\Z)_{\widetilde{j_b}(\overline{P})}$, allowing us to explicitly construct the map
\begin{align}\label{eqn:kappaExplicit}
    \kappa_{f,\Z}\colon \Z^r \to T_f(\Z)_{\widetilde{j_b}(\overline{P})}, \hspace{.2in} (n_1,\ldots, n_r) \mapsto D_{\tilde{t}}((p-1)n_1, \dots,  (p-1)n_r),
\end{align}

Finally, by \cite[Theorem~4.10]{EdixhovenLido}, the map $\kappa_{f,\Z}$ extends uniquely to a continuous map 
\begin{align}
\label{def:kappaf}
    \kappa_f\colon \Z_p^{r} \to T_f(\Z_p)_{\widetilde{j_b}(\overline{P})}.
\end{align}
The image of $\kappa_f$ is $\overline{T_f(\Z)}_{\widetilde{j_b}(\overline{P})}$.

By iterating over the basis $f_1, \dots, f_{\rho -1}$ of trace zero endomorphisms, we obtain the map
\begin{align}
\label{def:kappaz}
    \kappa_\Z \colon \Z^{r} \to T(\Z_p)_{\widetilde{j_b}(\overline{P})}
\end{align}
and its unique extension to a continuous map
\begin{align}
\label{def:kappa}
    \kappa\colon \Z_p^{r} \to T(\Z_p)_{\widetilde{j_b}(\overline{P})}.
\end{align}
The map $\kappa$ has image $\overline{T(\Z)}_{\widetilde{j_b}(\overline{P})}$. 

Recall the pseudoparametrization $\phi\colon T(\Z_p)_{\widetilde{j_b}(\overline{P})} \to \Z_p^{g} \times \Q_p^{\rho - 1}$ from Definition~\ref{def:paramT}.

\begin{proposition}
\label{prop:phicirckappa}
The map $\phi \circ \kappa\colon \Z_p^r \to \Z_p^{g} \times \Q_p^{\rho-1}$ is given by $g$ homogeneous linear polynomials and $\rho-1$ polynomials of degree at most $2$.
\end{proposition}
\begin{proof}
It is enough to show this for $\kappa_\Z$, since $\varphi\circ\kappa$ is continuous.

We make the identification $J(\Z)_{{j_b}(\overline{P})} = D_0 + J(\Z)_0 \simeq D_0 + \Z^r \simeq \Z^r$ where $D_0 \in J(\Z)_{{j_b}(\overline{P})}$. Under this bijection, the map $\phi \circ \kappa_\Z\colon J(\Z)_{\widetilde{j_b}(\overline{P})} \to \Z_p^{g} \times \Q_p^{\rho-1}$ is given by 
\[
D \mapsto \log(D - D_0),
\]
on the first $g$ components.  Since $\log$ is a group homomorphism, it follows the first $g$ polynomials are homogeneous linear as desired.

Now we fix one of the $\rho-1$ trace zero endomorphisms $f\colon  J \to J$. Let $\pi_f\colon  \Z_p^{g} \times \Q_p^{\rho-1} \to \Q_p$ be the projection onto the coefficient corresponding to $f$. Consider the map $\tau \coloneqq \pi_f \circ \phi \circ \kappa_\Z$. We write $F$ for the affine linear map \[\Z^r \simeq J(\Z)_{j_b(\overline{P})} \xrightarrow{f} J(\Z)_{\alpha(j_b(\overline{P}))} \simeq \Z^r\] where we identify $J(\Z)_{\alpha(j_b(\overline{P}))}$ with $\Z^r$ by subtracting $\alpha(D_0)$ and use $J(\Z)_0 \simeq \Z^r$.

By \cite[(4.2)-(4.4)]{EdixhovenLido} we have that $\tau(n_1,\dots,n_r)$ is a sum of a constant term, a linear function in the integers $n_1,\ldots,n_r$, a linear function in $Fn$ and a bilinear form evaluated in $(n,Fn)$. Since $F$ is linear, in total, this gives a function of degree at most $2$ in $n$.
\end{proof}

\section{The geometric quadratic Chabauty algorithm}
\label{sec:upperbound}

In this section, we present the main algorithm of this paper for doing geometric quadratic Chabauty. This algorithm ties together the results of the previous sections.

\begin{myalgorithm}[Geometric quadratic Chabauty in a single disk]
\label{alg:mainGQCalg}
\hfill

\noindent Input: 
\begin{itemize}
\item $X_\Q/\Q$ a smooth, projective, geometrically irreducible curve over $\Q$ such that $X_\Q(\Q) \neq \varnothing$ with a regular model $X$ of genus $g$ and Mordell--Weil rank $r$, and with Jacobian of N\'eron--Severi rank $\rho>1$, such that $r<g+ \rho - 1$;
\item $\rho - 1$ nontrivial independent trace zero endomorphisms represented by $(g \times g)$-matrices giving the action on the sheaf of differentials with respect to a fixed basis;
\item  an open set $U \subset X^\sm$ containing the smooth points of one geometrically irreducible component of $X_{\F_q}$ for all primes $q$;
\item  a prime $p>2$ of good reduction for $X$;
\item  a precision $k \in \N$;
\item  a base point $b \in X(\Z)$;
\item  a point $\overline{P}\in U(\F_p)$;
\item  a generating set $G_1,\ldots, G_{r'}$ of the Mordell--Weil group of $J$.
\end{itemize}
\noindent Output: $g+\rho-2$ integral convergent power series in $\Z_p\langle z_1,\dots,z_r \rangle$ up to precision $k$, defining $\widetilde{j_b}(U(\Z_p)_{\overline{P}}) \cap \overline{T(\Z)}$ inside $\overline{T(\Z)}$.

For each of the given trace zero endomorphisms $f$ do Steps 2 through 5.
\begin{enumalg}
    \item For each of the given trace zero endomorphisms $f$ do the following.
    \begin{enumalg}
    \item Compute the correspondence $A_\alpha$ that induces the endomorphism $m \cdot \circ f\colon J\to J$ as given in Lemma~\ref{lemma:A-alpha}. 
    \item Find the divisor representing $c=[(\calL_\alpha)_{b\times X}]\in J(\Z)$ using Algorithm~\ref{alg:computec}.
    \item Choose a local parameter $\nu$ to parametrize $U(\Z_p)_{\overline{P}}$ as $\nu\mapsto P_\nu$. By Proposition~\ref{prop:valuationheight} the map $\nu \mapsto \varphi\circ\lambda(\nu)$ is modulo $p^k$ given by a polynomial with bounded degree.  By calculating enough values, interpolate to find the polynomial expression. In particular, when $v=0$ and $p>3$, for $k=1$, the degree bound is $1$. In this case, compute $\varphi\circ\lambda(0), \varphi\circ\lambda(1)$ and interpolate the resulting line.
    \item With the generating set $G_1,\ldots, G_{r'}$, use Algorithm~\ref{alg:Qij} to compute points $Q_{ij}, Q_{i0} \in\calN$ up to precision $k$ that are the images of the generating sections of $\calM^\times(G_i, f(G_j))(\Z)$ and $\calM^\times(G_i,c)(\Z)$ for $1\le i,j\le r'$.
    \item \label{alg:step:kappaf} Using the elements $Q_{ij}$, find the map $\kappa_{f, \Z}\colon \Z^{r}\to T_f(\Z)_{\widetilde{j_b}(\overline{P})}$ as in \eqref{eqn:kappaExplicit} and extend it to the map $\kappa_f \colon  \Z_p^r \to T_f(\Z_p)_{\widetilde{j_b}(\overline{P})}$.
    \end{enumalg}
    \item Using the $\kappa_f$ for $f = f_1, \dots, f_{\rho-1}$ constructed in Step \eqref{alg:step:kappaf}, construct $\kappa$ as in \eqref{def:kappa}.
    \item Compose with the pseudoparametrization $\phi$ to compute the $g$ homogeneous linear and $\rho-1$ quadratic polynomials describing $\phi \circ \kappa\colon  \Z_p^r \to \Z_p^{g} \times \Q_p^{ \rho - 1}$, as guaranteed by Proposition~\ref{prop:phicirckappa}, up to precision $k$.
    \item Use Hensel lifting to compute the power series $g_1, \dots, g_{g+\rho-2}$ defined in Proposition~\ref{prop:embeddingcurve} that cut out $\im (\phi \circ \lambda)$, up to precision $k$.
    \item Return $g_i \circ (\phi \circ \kappa)$ for $i = 1, \dots, g+ \rho-2$.
\end{enumalg}
\end{myalgorithm}

By iterating this over all simple opens $U$ (as in Section~\ref{sec:overview}), and also iterating over all $\F_p$-points of $U$, we obtain multivariate power series up to precision $k$ cutting out $X(\Z_p)_{\Geo}$. 

\begin{remark}
By \cite[Section~9.2]{EdixhovenLido}, the power series in the output of Algorithm~\ref{alg:mainGQCalg} have at most finitely many zeros in $\Z_p$. In practice, one can solve these power series up to enough precision by using a multivariate Hensel's lemma \cite[Theorem~25]{KuhlmannHensel}. This assumes that the Jacobian matrix of the sequence of power series is invertible over $\Q_p$. We expect this to always happen unless there is a geometric obstruction.

Often solving these power series modulo $p$ is enough to determine $X(\Z_p)_{\Geo}$. See for example \cite[Theorem~4.12]{EdixhovenLido}, which we use in Section~\ref{sec:example}.
Even if computations modulo $p$ are not enough, one can increase the precision by considering the residue disks $U(\Z_p)_{\overline{P}}$, where $\overline{P} \in U(\Z/p^k\Z)$ for some integer $k$. An example of the geometric Chabauty method with higher precision is given in Remark~\ref{rem:gccubic}. 
\end{remark}

\begin{remark}
In practice, to run Algorithm~\ref{alg:mainGQCalg} we need to be able to compute Coleman--Gross heights on the curve $X$. Currently, this has only been made algorithmic for hyperelliptic curves.
\end{remark}

\section{The comparison theorem}
\label{sec:comparison}

In this section we give a comparison theorem between the geometric method and the cohomological quadratic Chabauty of  \cite{QCI,QCII,QCCartan,examplesandalg}. In Theorem~\ref{thm:comparison}, we show that the geometric method produces a refined set of points, as is the case for classical Chabauty--Coleman \cite{geometricLinearChab}. 

For this section we assume that $p$ is a prime of good reduction, that $r = g$, that $\rho > 1$, and further, that $\overline{J(\Z)}$ has finite index in $J(\Z_p)$. The cohomological quadratic Chabauty set in \cite{QCI} is defined under these assumptions. 
We do not require a semistable model for $X/\Q_q$, $q|n$ as is sometimes assumed; a semistable model can make explicit calculations of heights away from $p$ easier, see \cite{BettsDogra} or \cite[Section~3.1]{examplesandalg}. By \cite[Lemma~6.1.1]{BettsEffective} the local heights away from $p$ factor through the component set of the minimal regular model.

Let $Z_1, \dots, Z_{\rho - 1}$ be a basis for $\ker(\NS(J) \to \NS(X))$. In the cohomological method, from the transpose $Z_i^\top$ of such a correspondence\footnote{Due to a difference of conventions of rigidifications for line bundles on $X \times X$, we have to take the transpose of $Z_i$ for the methods to align perfectly. The transpose $Z_i^\top$ induces the same endomorphism of the Jacobian.} we can construct a quadratic Chabauty function $\sigma_i\colon  X(\Q_p) \to \Q_p$ and a finite subset $\Omega_i \subset \Q_p$ described explicitly in terms of local heights at primes of bad reduction such that $\sigma_i(z) \in \Omega_i$ for all $z \in X(\Q)$. This finite subset $\Omega_i$ consists of one constant $c_{U,i}$ for every simple open $U$.

We describe the construction of $\sigma_i$ and the set $\Omega_i$ in more detail after we present the main theorem.
The divisor $Z_i$ is the correspondence of a trace zero endomorphism $f_i\colon  J \to J$ of the Jacobian. In the geometric method, we work with the endomorphism $ \alpha_i \coloneqq m \cdot \circ  \tr_{c_i }  \circ f_i$. This multiplication with $m$ will result in all the heights in the trivial biextension $\calN$ to be a factor $m$ larger than in the cohomological case.

\begin{definition}
\label{def:qcset}
Define $X(\Q_p)_{\Coh} \coloneqq \bigcup_{U} \{x \in X(\Q_p) \mid  \sigma_i(x) = c_{U,i}, \text{ for } i = 1, \dots, \rho -1\}$ where the union is over all simple opens $U$.
\end{definition}

\begin{remark}
\label{rem:definitionqcoh}
As far as we know, the existing literature does not explicitly define the quadratic Chabauty set in the case of multiple endomorphisms. In the case where one uses a single trace zero endomorphism, the set is defined in \cite[Theorem~1.2]{QCI}.
One can see Definition~\ref{def:qcset} as a special case of the finite set implicitly defined in \cite[Theorem~A]{BettsEffective}, for the quotient of the fundamental group that is an extension of the abelianization by $\Q_p(1)^{\rho-1}$.

The alternative definition is $\bigcap_i \bigcup_U \{x \in X(\Q_p) \mid \sigma_i(x) = c_{U,i}\}$. Here the union and the intersection have been switched, and hence the resulting set can be bigger. The difference between the two sets consists exactly of points $x \in X(\Q_p)$ such that $\sigma_i(x) \in \Omega_i$ for every $i$, but such that there is no $U$ with $\sigma_i(x) = c_{U,i}$ for every $i$. In particular, the points in the difference do not lie in any of the simple opens $U$, and hence are not rational points.
\end{remark}

Recall the definition of $X(\Z_p)_{\GQC}$ from Definition~\ref{def:xzpgeo}. Given a covering of $X(\Z)$ by simple opens $U$ we have that 
\[X(\Z_p)_{\GQC} \colonequals \bigcup_{U}  \widetilde{j_b}^*(\widetilde{j_b}(U(\Z_p)) \cap \overline{T(\Z)}) \subset \bigcup_{U} U(\Z_p) = X(\Z_p). \]

The following definitions give terminology for two of the cases in which $X(\Q_p)_{\Coh}$ is strictly bigger than $X(\Z_p)_{\GQC}$. 
\begin{definition}
\label{def:goodreduction}
We say that the Mordell--Weil group is \defi{of good reduction} (modulo $p$) if the map $\overline{J(\Z)}_{0}/p\overline{J(\Z)}_0 \to J(\Z/p^2\Z)_{0}$ is injective. Otherwise, we say that it is \defi{of bad reduction}.
\end{definition}

The Mordell--Weil group being of good reduction is equivalent to the map $\ol{J(\Z)}_0 \to J(\Z_p)_0$ being an isomorphism. On the level of abstract groups, this map is always an embedding $\Z_p^g \to \Z_p^g$ with image of index some power of $p$. Another equivalent way of stating this is that the 
\defi{$p$-saturation of $\ol{J(\Z)}_0$ in $J(\Z_p)_0$} \[\{x \in J(\Z_p)_0 \mid \exists k, p^k x \in \ol{J(\Z)}_0\}\]  is always equal to $J(\Z_p)_0$, and the Mordell--Weil group is of bad reduction if and only if this $p$-saturation is bigger than $\ol{J(\Z)}_0$.

\begin{definition}
For $Q \in X(\F_p)$, if $j_b(Q)$ is not in the image of the reduction map $J(\Z) \to J(\F_p)$, then we say $Q$ \defi{fails the Mordell--Weil sieve} (at $p$). In this case, the residue disk $X(\Z_p)_Q$ cannot contain a rational point. Otherwise, $Q$ \defi{passes the Mordell--Weil sieve} (at $p$).
\end{definition}

Our main theorem is the following comparison theorem.
\begin{theorem}
\label{thm:comparison}
There is an inclusion $X(\Q) \subseteq X(\Z_p)_{\GQC} \subseteq X(\Q_p)_{\CQC}$. For $P \in X(\Q_p)_{\CQC}$ we have $P \not \in X(\Z_p)_{\GQC}$ if and only if one of the following conditions holds:
\begin{enumerate}
  \item \label{failMWsieve} $P$ fails the Mordell--Weil sieve at $p$;
  \item \label{badreduction} the Mordell--Weil group is of bad reduction at $p$ and $j_b(P)$ does not lie in the $p$-adic closure of the Mordell--Weil group, but only in its $p$-saturation. 
\end{enumerate}
\end{theorem}

\begin{remark}\label{rem:manyEndomorphisms}
It follows immediately from the proof of Theorem~\ref{thm:comparison} that the inclusion $X(\Q) \subseteq X(\Z_p)_{\GQC} \subseteq X(\Q_p)_{\CQC}$ and comparison from Theorem \ref{thm:comparison} also hold when the sets $X(\Z_p)_{\GQC}$ and $X(\Q_p)_{\CQC}$ are constructed using a fixed subset $Z_{i_1}, \dots, Z_{i_k}$ of $1 \leq k < \rho -1$ independent elements of $\ker(\NS(J) \to \NS(X))$, instead of a full basis.
\end{remark}

\begin{remark}
In \cite{geometricLinearChab}, an analogous theorem is given for the comparison between the classical Chabauty--Coleman method, as in \cite{EffectiveChab, BBK}, and the geometric linear Chabauty, as developed in \cite{PimThesis} and \cite{geometricLinearChab}. The comparison theorem \cite[Theorem~4.1]{geometricLinearChab} shows that the set of candidates found by the classical Chabauty--Coleman method contains the set found by geometric linear Chabauty method. Furthermore, the two sets differ by conditions analogous to ~\eqref{failMWsieve} and \eqref{badreduction}.
\end{remark}

Let $1 \leq i \leq \rho(J) - 1$.
We briefly recall the constructions of $\sigma_i$ and $ \Omega_i$ from \cite{examplesandalg}. For more details, the reader can also consult \cite{QCI,QCCartan}. The cohomological method for quadratic Chabauty uses Nekov\'a\v{r}'s theory \cite{NekovarHeights} of $p$-adic heights of certain Galois representations to construct a global height $h_i^{\Nek}\colon X(\Q) \to \Q_p$ by attaching a family of Galois representations to $X(\Q)$ and $X(\Q_p)$. The Galois representation depends on the choice of base point $b$ as well as the correspondence $Z_i$. We suppress this dependence on $b$ in our notation. The global height also depends on a choice of splitting of the Hodge filtration and id\`ele class character, which we choose to be compatible with the choices made to construct the Coleman--Gross height $h$. In particular we choose the cyclotomic character.
This global height $h_i^{\Nek}$ factors through $h^{\Nek}\colon J(\Q) \times J(\Q) \to \Q_p$ \cite[Section~2.3]{examplesandalg}. We can thus extend $h^{\Nek}$ on $J(\Q) \times J(\Q)$ to a bilinear function on $J(\Q_p) \times J(\Q_p) \to \Q_p$ and evaluate it on elements of $X(\Q_p)$.

This global height decomposes as a sum of local heights over finite places
\[ h_i^{\Nek} = \sum_v h_{i,v}^{\Nek}\] where $h_{i,v}^{\Nek} \colon X(\Q_v) \to \Q_p$.
Define the quadratic Chabauty function
\[\sigma_i (z) \colonequals h_i^{\Nek}(z) - h_{i,p}^{\Nek}(z)\]
 for $z \in X(\Q_p)$, recalling that the right hand side implicitly depends on $Z_i$. 
 Then, for any $z \in X(\Q)$, using the decomposition above we can write $h_i^{\Nek}(z) = h_{i,p}^{\Nek}(z) + \sum_{q \neq p} h_{i,q}^{\Nek}(z)$. 
 The set $\Omega_i \subset \Q_p$ is defined by the local heights in the following way. Let
\begin{align*}
\Omega_{i, q} \colonequals \{ h_{i,q}^{\Nek}(z)  \mid z \in X(\Q_q) \}.
\end{align*}
If $X_{\F_q}$ is geometrically irreducible, then $\Omega_{i,q} = \{0\}$.
We can therefore define the finite set
\begin{align}
\label{eq:omegai}
\Omega_i \colonequals \{ \sum_q w_q \mid w_q \in \Omega_{i,q} \} ,
\end{align} 
 Hence, when $z \in X(\Q)$, we have $\sigma_i(z) \in \Omega_i$ and so $X(\Q_p)_{\Coh} \supseteq X(\Q)$.

\begin{remark}
The function $\sigma_i(z)$ is locally analytic \cite[pp. 6, 10]{examplesandalg}. If $X$ has sufficiently many rational points, then one can explicitly express the function $\sigma_i(z)$ as a power series in every residue disk, and for each $c \in \Omega_i$ and each residue disk of $X(\Q_p)$ find the roots of $\sigma_i(z) - c$ to explicitly solve for elements of $X(\Q_p)_{\Coh}$. 
\end{remark}

The following theorem relates the local height of the Galois representation associated to a point $P \in X(\Q_p)$ to a pairing with a divisor that is studied in \cite{DarmonRotgerSols}.
\begin{definition}
Let $z \neq b$ be a point in $X(\Q_p)$.
Define $D_{Z_i^\top}(z, b)$ to be the degree zero divisor on $X$ given by $D_{Z_i^\top}(z, b) \colonequals {Z_i}|_{\Delta} - {Z_i}|_{X \times b} - {Z_i}|_{z \times X} $. 
\end{definition}

\begin{theorem}{\cite[Theorem~6.3]{QCI}} 
\label{thm:QCInektoCG}
Let $q$ be a prime and let $z \neq b$ be a point in $X(\Q_p)$.
We have the equality of local heights $h^{\Nek}_{i,q}(z)  =h_{q} (z-b, D_{Z_i^\top}(z,b))$ and moreover $h_i^{\Nek}(z)  =h (z-b, D_{Z_i^\top}(z,b))$ where $h$ is the Coleman--Gross height.
\end{theorem}

\begin{proposition}
\label{prop:DbzAalpha}
Let $z \in X(\Z_p)$ be such that $z\neq b$. We have $-mD_{Z_i^\top}(z,b) = A_{\alpha_i}|_{z \times X}$ and $-m[D_{Z_i^\top}(z,b)] = [\alpha_i(z-b)]$.
\end{proposition}
\begin{proof}
Write $B = {Z_i}|_{X \times b}$ and $C = {Z_i}|_\Delta$. Then \[D_{Z_i^\top}(z,b) = C - B - {Z_i}|_{z \times X}.\]  Define $A = Z_i - B \times X + X \times B - X \times C$. Then we see $A|_{z \times X} = {Z_i}|_{z \times X} + B - C = -D_{Z_i^\top}(z,b)$. Then by Lemma~\ref{lemma:A-alpha}, $mA$ is equal to $A_{\alpha_i}$ and the proposition follows. 
\end{proof}

\begin{definition}
Define $\rho_\calN\colon \calN \to \Q_p$ by $(D_1,D_2,x) \mapsto h(D_1,D_2) - x$.
\end{definition}
Note that $\rho_{\calN}$ does not depend on $Z_i$.

\begin{lemma}
\label{lem:rhonvanish}
The function $\rho_\calN$ vanishes on the image $\Psi(\calM^\times(\Z))$ in $\calN$, and in particular, on $\Psi(T_i(\Z))$.
\end{lemma}
\begin{proof}
This follows from Proposition~\ref{prop:intpointscalN}.
\end{proof}

In order to characterize the difference between $X(\Z_p)_{\Geo}$ and $X(\Q_p)_{\Coh}$, we will use the following lemma. 

\begin{lemma}
\label{lem:diffzarclos}
The difference $Z(\rho_{\calN}) \setminus\Psi(\overline{\calM^\times(\Z)})$ consists of all the points $(D,E,x)$ with $D,E \in J(\Z_p)$ such that
\begin{enumerate}
\item $D$ or $E$ fail the Mordell--Weil sieve, or;
\item the Mordell--Weil group is of bad reduction, and at least one of $D$ or $E$ does not lie in the $p$-adic closure $\overline{J(\Z)}$ of the Mordell--Weil group, and only lies in its $p$-saturation.
\end{enumerate}
\end{lemma}
\begin{proof}
Note that $Z(\rho_{\calN})$ is in bijection with $J(\Z_p) \times J(\Z_p)$. In contrast, $\Psi(\calM^\times(\Z))$ is in bijection with $J(\Z) \times J(\Z)$. By assumption, $\overline{J(\Z)}_0 \subset J(\Z_p)_0$ is a finite index $\Z_p$-sub-module, and therefore has $p$-saturation $J(\Z_p)_0$. Hence $J(\Z_p) \setminus \overline{J(\Z)}$ consists exactly of points failing the Mordell--Weil sieve and points that only lie in the $p$-saturation of $\overline{J(\Z)}$ and not in $\overline{J(\Z)}$ itself.
This can only happen if $\overline{J(\Z)_0}$ is a proper subgroup of $J(\Z_p)_0 \simeq \Z_p^g$. A finite index $\Z_p$-submodule $G \subset \Z_p^g$ is a proper subgroup if and only if after tensoring with $\F_p$ the induced map $G/pG \to \F_p^g$ is not an isomorphism. This is equivalent to $G/pG \to \F_p^g$ not being injective. So the second condition can only happen if the Mordell--Weil group is of bad reduction.
\end{proof} 

\begin{definition}
\label{def:cui}
Let $U \subset X^{\sm}$ be a simple open set of $X^{\sm}$. Define $c_{U,i} \in \Omega_i \subset \Q_p$ to be $\sum_{q \neq p} h_q(z_q-b, D_{Z_i^\top}(z_q,b))$ for any $z_q \in U(\Z_q)$ with $z_q\ne b$.
\end{definition}

\begin{remark}
By Lemma~\ref{lem:Vq}, this is well defined and by this lemma as well as Proposition~\ref{prop:DbzAalpha} we have that $m c_{U,i}$ is equal to $\sum_q V_q \log q$, with $V_q$ as defined in Definition~\ref{def:Vq}. Hence by \eqref{eq:psijbtilde} we have that $\psi \circ \widetilde{j_b}\colon U(\Z_p) \to \Q_p$ is given by $z \mapsto h_p(z-b,A_\alpha|_{z \times X}) - mc_{U,i}$ and $(\Psi \circ \widetilde{j_b})(z) = (z-b,A_{\alpha}|_{z \times X}, h_p(z-b,A_\alpha|_{z \times X}) - mc_{U,i})$.
\end{remark}

\begin{lemma}
\label{lem:pullbackrhon}
The function $-m(\sigma_i(z) - c_{U,i})$ is the pullback along $\Psi \circ \widetilde{j_b}|_{U}\colon U(\Z_p) \to \calN$ of $\rho_{\calN}$.
\end{lemma}
\begin{proof}
Let $z \in U(\Z_p) \subset X(\Z_p)$ with $z\neq b$. By Theorem~\ref{thm:QCInektoCG} and Proposition~\ref{prop:DbzAalpha} we have that 
\[-mh_{i,q}^{\Nek}(z) = -mh_q(z-b,D_{Z_i^\top}(z,b)) = h_q(z-b,A_\alpha|_{z  \times X}).\] 
By Lemma~\ref{lem:Vq}, 
\[h_q(z-b,A_\alpha|_{z  \times X}) = -V_q \log q\].

Then 
\begin{align*}
&-m(\sigma_i(z) - c_{U,i}) = -m(h^{\Nek}(z) - h_p^{\Nek}(z) - c_{U,i})\\
& = h(z-b,A_\alpha|_{z \times X}) - h_p(z-b,A_\alpha|_{z \times X}) + m c_{U,i}.
\end{align*} This is equal to
\begin{align*}
&h(z-b,A_\alpha|_{z \times X}) - (h_p(z-b,A_\alpha|_{z \times X}) - m c_{U,i}) =\\
& \rho_{\calN}((z-b,A_{\alpha}|_{z \times X}, h_p(z-b,A_\alpha|_{z \times X}) - mc_{U,i})) 
=  \rho_{\calN}(\widetilde{j_b}(z)).
\end{align*}
This last equality follows from Corollary~\ref{cor:jbtilde}.
\end{proof}

\begin{proof}[Proof of Theorem~\ref{thm:comparison}]
Let $c \in \Omega_i$, and consider the function $\sigma_i - c$. 
By \eqref{eq:omegai}, Theorem~\ref{thm:QCInektoCG}, and Definition~\ref{def:cui} there is a simple open $U \subset X$ such that $c = c_{U,i}$.  

Let $\widetilde{j_{b,U,i}}$ denote the map $U \to T_i$. According to Lemma~\ref{lem:pullbackrhon} we have that $-m(\sigma_i -c)\colon U(\Z_p) \to \Q_p$ is the composite
\begin{equation}
\label{eq:gui}    
U(\Z_p) \xrightarrow{\widetilde{j_{b,U,i}}} T_i(\Z_p) \to \calM^\times(\Z_p) \xrightarrow{\Psi} \calN \xrightarrow{\rho_{\calN}} \Q_p,
\end{equation} where $T_i(\Z_p) \to \calM^{\times}(\Z_p)$ is the natural injective map. Define $g_{U,i} \colonequals -m(\sigma_i -c)$. Note that the first three maps in \eqref{eq:gui} are injections.

With this formulation we have \[X(\Q_p)_{\Coh} = \bigcup_{U} \bigcap_{i}  Z(g_{U,i}).\] 

Similarly, we can write
\[
X(\Z_p)_{\Geo} = \bigcup_{U} \bigcap_i \widetilde{j_{b,U,i}}^*(\widetilde{j_{b,U,i}}(U(\Z_p)) \cap \overline{T_i(\Z)}).
\]

By Lemma~\ref{lem:rhonvanish}, the set $Z(g_{U,i})$ contains \[\widetilde{j_{b,U,i}}^*(\widetilde{j_{b,U,i}}(U(\Z_p)) \cap \overline{T_i(\Z)}).\]
Therefore, we get the containment $X(\Z_p)_{\Geo} \subseteq X(\Q_p)_{\Coh}$.

By Lemma~\ref{lem:diffzarclos} for fixed $U,i$ the difference \[Z(g_{U,i}) \setminus \widetilde{j_{b,U,i}}^*(\widetilde{j_{b,U,i}}(U(\Z_p)) \cap \overline{T_i(\Z)})\] consists exactly of points $P$ that fail the Mordell--Weil sieve and points $P$ such that $j_b(P)$ lies not in $\ol{J(\Z)}$ but only in its $p$-saturation. We see that an element of $X(\Q_p)_{\Coh} \setminus X(\Z_p)_{\Geo}$ satisfies condition~\eqref{failMWsieve} or condition~\eqref{badreduction} of Theorem~\ref{thm:comparison}.

On the other hand, if $P \in X(\Q_p)_{\Coh}$ fails the Mordell--Weil sieve or $j_b(P) \notin \ol{J(\Z)}$, then $P \not \in X(\Z_p)_{\Geo}$. The theorem follows.
\end{proof}

\section{Example}
\label{sec:example}

We give an example of the implementation on  the modular curve $X_0(67)^+$ of the algorithms presented. The rational points on this curve have already been determined \cite{RecentApproaches} using quadratic Chabauty and a Mordell--Weil sieve, but we can also use the methods presented here to show the following proposition about the rational points of the curve in one residue disk. \texttt{Magma} code that can be used to verify the computations here can be found in \cite{DRHSgitrepo}. 
Let $X$ be a regular model for $X_0(67)^+$ over the integers given by the homogenization of $y^2 + (x^3 + x+ 1)y = x^5 - x$ in the weighted projective plane $\PP^2_{(1, 3,1)}$. Then $X(\Q) = X(\Z)$ and we show the following.
\begin{theorem}
\label{thm:mainthm67}
The integer points of $X(\Z)$ that do not reduce to $(1, 4) \in X(\F_7)$ are contained in the set
\begin{align*}
&\{[0:-1:1],[4\cdot 7 +  O(7^2): 6 + 6\cdot 7 + O(7^2):1],[0:0:1],[4\cdot7 + O(7^2) : 3\cdot7 + O(7^2) : 1],\\
&[1:0:1],[1 + 2 \cdot 7 + O (7^2): 5\cdot 7+ O(7^2):1],[1:-3:1],[1+2\cdot7 + O(7^2) : 4 + O(7^2) : 1],\\
&[1:-1:0],[1:6 + 3\cdot 7+O(7^2):3\cdot 7+O(7^2) ], [1:0:0],[1:4\cdot 7+O(7^2):4\cdot 7+O(7^2)]\}.
\end{align*}
\end{theorem}

\begin{remark}
\label{rem:residuediskR}
The residue disk above $(1, 4) \in X(\F_7)$ has at least two integer points, $[1:-3:2]$ and $[1:-10:2]$. Using geometric quadratic Chabauty modulo $p^2$, we cannot bound the size of this residue disk. After doing the necessary calculations, it turns out $\im \widetilde{j_b}(z) = \im \kappa(0, n_2)$. 
In this case, applying \cite[Theorem~4.12]{EdixhovenLido}, since the ring $\F_p[n_1, n_2]/(\overline{g_1}, \overline{g_2} ) \simeq \F_p[n_2]$ is not finite, we cannot determine the solutions using calculations modulo $p^2$.

By increasing precision we are guaranteed a finite set of solutions in this residue disk. In practice, this requires computing heights of points that lie in residue disks at infinity which is not possible using current implementations of Coleman--Gross heights.
\end{remark}

We present the computations in a single residue disk over $\overline{P} = (0, -1) \in X(\F_7)$ where we show the following.
\begin{proposition}
\label{prop:points}
The integer points of $X(\Z)$ reducing to $(0, -1) \in X(\F_7)$
are contained in the set
\[\{(0,-1),(4\cdot 7 +  O(7^2), 6 + O(7^2)) \}.\]
\end{proposition}

We first list some facts about this curve that will be useful in our computations. The curve $X$ is a projective curve of genus $2$ with Jacobian $J$. We recall some details about $X$ and its Jacobian that are presented in \cite[Section~6]{RecentApproaches}. The Jacobian $J$ has Mordell--Weil rank $2$ and $J_\Q$ has N\'eron--Severi rank $2$.  In addition, the only prime of bad reduction of $X$ is $67$. At $67$, the special fiber is geometrically irreducible: it has one component with two nodes defined over $\F_{67^2}$. Hence, there are only geometrically irreducible fibers over every prime.
\begin{remark}
For this example curve, all of the fibers are geometrically irreducible, leading to a simplification in the notation used in the example compared to the notation in the preceding sections. In general, one needs to consider a distinction between $J$ and $J^0$, where $J^0$ is the fiberwise connected component of $0$ in $J$.
We also omit the constant $m$ which is the least common multiple of the exponents of all $J/J^0(\overline{\F}_p)$, with $p$ ranging over all primes. Since $J = J^0$, we have $m = 1$. Let $X^\sm$ denote the open subscheme of $X$ consisting of points at which $X$ is smooth over $\Z$. Above, we consider the simple open subschemes $U$ of $X^{\rm sm}$. In this example, there is only one simple open to consider: the scheme $X^\sm$ obtained by removing the two Galois conjugate nodes in the fiber over $67$. Since $X$ is regular, $X^\sm(\Z) = X(\Z)$.
\end{remark}

Let $\iota$ be the hyperelliptic involution of $X$. We list some rational points on the curve that will be used in our computations:
\begin{align}
\label{eqn:namesofpoints}
    P&\colonequals[0:-1:1], &\iota P&\colonequals[0:0:1],\notag
    \\Q&\colonequals[-1:0:1],&\iota Q&\colonequals[-1:1:1],\notag\\
    b&\colonequals[1:0:1],&\iota b&\colonequals[1:-3:1],\\
    R&\colonequals[1:-3:2],&\iota R&\colonequals[1:-10:2],\notag\\
    \infty_+&\colonequals[1:0:0], &\infty\,\!_ -&\colonequals[1:-1:0].\notag
\end{align}
These points turn out to be the only rational points on $X$, as proven in \cite[Theorem~6.3]{RecentApproaches} by a combination of quadratic Chabauty and the Mordell--Weil sieve. 

Let $p = 7$. We first perform some local computations. There are $9$ points on $X(\F_p)$. For each $\F_p$-point $x$ of $X^\sm$, we need an element in $T(\Z)_{\widetilde{j_b}(x)}$, or equivalently an element in $J(\Z)_{j_b(x)}$. Every residue disk of $X(\Z_p)$ contains an integer point; only $R$ and $\iota R$ reduce to the same point. Therefore, none of the residue disks $J(\Z)_{j_b(x)}$ are empty. So we cannot rule out any residue disks of the torsor immediately; in fact, this calculation is a Mordell--Weil sieve at $p$, see \cite[Section~3.4]{geometricLinearChab} for more details.

This example presents the specific case of the residue disk corresponding to $X(\Z)_{\overline{P}}$, where $P$ is the point defined in \eqref{eqn:namesofpoints}. Because we can consider residue disks up to the hyperelliptic involution, this also gives us the analogous result for the residue disk corresponding to $\iota P$.

Let $j_b\colon X^\sm\to J$ denote the Abel--Jacobi map with base point $b$ defined in \eqref{eqn:namesofpoints}. We also have a set of generators for the Mordell--Weil group $J(\Z)$ from the \href{https://www.lmfdb.org/Genus2Curve/Q/4489/a/4489/1}{LMFDB}, 
\begin{align}\label{eqn:generatorMordellWeil}
    G_1 &\colonequals [P-\iota P],\\\notag
    G_2 &\colonequals [P + Q - 2\cdot \iota P].
\end{align}

Since $X$ is a modular curve, its Jacobian has an action by the Hecke algebra. To describe the Hecke action on $J$ explicitly, we fix the following basis for $H^0(X_\Q, \Omega_{X_\Q}^1)$ :
\begin{align}\label{eqn:basisOneForms}
    \left\{\frac{dx}{2y  -x^3 - x - 1},\;\;\frac{xdx}{2y  -x^3 - x - 1}\right\}.
\end{align}
We focus on the endomorphism given by the action of the Hecke operator $T_2$ on 1-forms of $X$. The Kodaira--Spencer map gives an isomorphism between $H^0(X_\Q, \Omega_{X_\Q}^1)$ and $S_2(67)^+$. We choose a basis for $S_2(67)^+$ that is given by $q$-expansions with rational coefficients, as follows:
\begin{align*}
    g_1&\colonequals q - 3q^3 - 3q^4 - 3q^5 + q^6 + 4q^7 + 3q^8 +O(q^9),\\
    g_2&\colonequals q^2 - q^3 - 3q^4 + 3q^7 + 4q^8 +O(q^9). \notag 
\end{align*}
Then we choose the model for $X$ where $\frac{du}{v}$ corresponds to $g_1 \frac{dq}{q}$ and $u\frac{du}{v}$ corresponds to $g_2 \frac{dq}{q}$, by setting $u = \frac{g_2}{g_1}$ and $v = q \frac{du}{g_1 dq}$. This allows us to find $q$-expansions for the monomials $\{ v^2, 1, u, u^2, \ldots, u^5, u^6\}$ and use linear algebra to get an explicit equation for the new model of $X$,
\begin{align*}
    v^2=9u^6 -14u^5 + 9u^4 -6u^3 + 6u^2 -4u + 1.
\end{align*}

Writing down an explicit change of model to the regular model, we can find the $q$-expansion of the forms in \eqref{eqn:basisOneForms} and compute the Hecke action on these $q$-expansions.
This gives us the matrix representation of the Hecke operator $T_2$ with respect to the basis on \eqref{eqn:basisOneForms}. The trace of this matrix is nonzero, so we let $f\colonequals 2T_2 + 3 \id\colon J \to J$. The endomorphism $f$ has trace zero and matrix representation
\begin{align}\label{eqn:matrixRepf}
    \begin{pmatrix}
    1&-2\\-2&-1
    \end{pmatrix}
\end{align}
with respect to the basis presented in (\ref{eqn:basisOneForms}).
Using the work of \cite{RigorousEndo}, we can compute a divisor $D_f \subset X_{\Q} \times X_{\Q}$ inducing $f$. The equations that define this divisor are given in \eqref{eqn:DfExample}. Then Algorithm~\ref{alg:Aalpha} produces the divisor $A_\alpha$ that satisfies the properties of Lemma~\ref{lemma:A-alpha}.

We now use Algorithm~\ref{alg:applyf} to calculate $f(G_1)$ and $f(G_2)$, where $G_1$ and $G_2$ are the generators of the Mordell--Weil group of $J$ as in \eqref{eqn:generatorMordellWeil}.

Since $J(\Z) = J(\Q)$, the divisor $f(G_i)$ only needs to be computed over the rationals for $i = 1,2$. For example, applying  \eqref{eqn:alphaPminQ} we get $f(G_1)=\calO_X(D_f|_{P \times X} - D_f|_{\iota(P) \times X})$ and we can compute an explicit divisor $f(G_1)$ using the equations for $D_f$. We find that 
\begin{align}\label{eqn:fOfG1G2}
f(G_1) &= -G_1+2 G_2 = [-(P-\iota P) + 2(P+Q - 2\iota P)] = [P + 2Q - 3\iota P],\\ 
f(G_2) &= 2G_1 +G_2 = [2(P-\iota P) + 1(P+Q - 2\iota P)] = [3P + Q -4\iota P].\notag
\end{align}
Furthermore, we compute $c= [-11G_1-8 G_2]$ using Algorithm~\ref{alg:computec}.

We can parametrize the residue disk over $\overline{P}$ up to finite precision by
\begin{align}
    \label{eq:param}
    \F_p &\to X(\Z/p^2\Z)_{\overline{P}},  \hspace{0.2in} \nu \mapsto P_\nu \text{ such that } x(P_\nu)/p = \nu.
\end{align}

We now find the trivializing section $\phi \circ \lambda$, following Section~\ref{sec:embeddingcurve}. By direct computation the constant $v$ from Proposition \ref{prop:valuationheight} is $0$, hence the pseudoparametrization $\phi$ has codomain $\Z_p^3$ (instead of $\Z_p^2 \times \Q_p$).  This computation is done using code from the repository \cite{QCMod}.

Since $p > 3$, by Proposition~\ref{prop:valuationheight} the map $\phi \circ \lambda: \Z_p \to \Z_p^3$ is linear  modulo $p$. We will calculate $\widetilde{j_b}(P_0) $ and $\widetilde{j_b}(P_1)$ following Algorithm~\ref{alg:jbtilde} and interpolate to determine the map. What the following computations show is that 
\begin{equation}\label{eqn:jbtildeEx}
\phi \circ \lambda(\nu) \equiv (2\nu,0,6-\nu) \mod p.
\end{equation} 
By Proposition~\ref{prop:embeddingcurve}, the image of the map $\phi \circ \lambda$ is cut out by two convergent power series. Giving $\Z_p^3$ the coordinates $(x_1, x_2, x_3)$, we see the image of $\phi \circ \lambda$ is cut out by the equations $g_1 = 0,g_2 = 0$ with  $g_1 \equiv x_2 \mod p,\, g_2 \equiv 2x_3 + x_1 + 2 \mod p$.

Algorithm~\ref{alg:jbtilde} relies on being able to compute Coleman--Gross local heights at $p$ and at primes of bad reduction. We first note that, since the special fiber of $X$ at $67$ is geometrically irreducible, the heights at $\ell \neq p$ are all trivial, and we only have to consider the heights at $p$.
Balakrishnan \cite{BalaLocalhts} has implemented Coleman--Gross local heights $h_p(D, E)$ for disjoint divisors of degree $0$ on a curve $Y$ with a few requirements:
\begin{enumerate}
    \item \label{req:odd-deg} the hyperelliptic curve $Y\colon y^2 = H(x)$ is given by a monic odd degree model;
    \item \label{req:splitting} the divisors $D$ and $E$ split as a sum of points $D= \sum_i n_i P_i$, $E= \sum_j m_j Q_j$  with $P_i, Q_j \in Y(\Q_p)$.
\end{enumerate}
\begin{remark}
\label{rem:linequiv}
Suppose that $D= \sum_i n_i P_i$ and $E  =  \div r +E'$  where $E' = \sum_j m_j Q_j$ with $P_i, Q_j \in Y(\Q_p)$. Then \[h_p(D, E) = h_p(D, E' + \div r) =  h_p(D, E') + h_p(D, \div r) = h_p(D, E')  + \log( r(D) )\] so we can also compute $h_p(D, E)$.
\end{remark}

Therefore we make a change of model when doing computations on $\calN$. The even degree model of $X$ is given by
\begin{align*}
    y^2 =  g(x)\colonequals x^6 + 4x^5 + 2x^4 + 2x^3 + x^2 - 2x + 1,
\end{align*}
where $g(x)$ has a $7$-adic zero $\beta = 4 + 3\cdot 7 + 4\cdot 7^2 + O(7^3)$. We can construct a degree $5$ model:
\begin{align*}
    \beta^6 y'^2 = g(\beta x'/(x'-1)) \cdot (x'-1)^6.
\end{align*}
Letting $c_0= 5 + 3\cdot 7 + 3\cdot 7^2 + O(7^3)$ be a $5$th root of the leading coefficient of $g(\beta x'/(x'-1))$ we obtain an odd degree model over $\Q_p$ given by the coordinate transformation from the even degree model
\begin{align}
\label{eqn:modelchange}
    (x,y) \mapsto (c_0 \cdot x/(x - \beta), \beta^3 y/(x-\beta)^3).
\end{align}

\begin{remark}
Forthcoming work of Gajovi\'{c} gives a practical algorithm and code for computing Coleman--Gross local heights $h_p(D, E)$ on even degree hyperelliptic curves.
\end{remark}

We now compute for $P$ the local height $\psi(\widetilde{j_b}(P)) = h_p(P-b,A_\alpha|_{P \times X})$. Let $B,C$ be the divisors on $X$ defined in Algorithm~\ref{alg:Aalpha}. One can check that $B\cap P_{\nu}$ is empty over $\Z/p^2\Z$ for all $\nu \in  \F_p$, so we have 
$ A_\alpha|_{P_\nu \times X} = D_f|_{P_\nu \times X} + B - C$; we denote $A_\alpha|_{P_0 \times X} $ by  $E_{P_0}$. Over the rationals 
\begin{align*}
    E_{P_0} \sim [0 : 0 : 1] -[-1 : 1 : 1] +2[-1 : 0 : 1] -2[1 : -3 : 1]  =: E_{P_0}',
\end{align*}
with $E_{P_0} = E_{P_0}' + \Div g_{P_0}$ where $g_{P_0}$ is computed explicitly as an element of the function field and given by equation \eqref{eqn:P0}.
By Remark~\ref{rem:linequiv}, we can decompose $h_p(P-b, E_{P_0}) = h_p(P-b,E_{P_0}') + h_p(P-b, \Div g_{P_0})$. We compute 
\begin{align*}
    h_p(P-b, \Div g_{P_0}) = \log g_{P_0}(P)/g_{P_0}(b) = \log (4/9) \equiv 7 \mod 49.
\end{align*}
We also compute
\begin{align*}
    h_p(P-b,E_{P_0}') =5\cdot 7 + 3\cdot7^2 + 3\cdot7^3 + 6\cdot7^4 + 7^5 + 5\cdot7^6 + 2\cdot7^7 + 6\cdot7^8 + O(7^9).
\end{align*}
So, $\psi(\widetilde{j_b}(P)) = 6 \cdot 7 + O(7^2)$.

Unlike the $P_0$ case, the divisor $D_{P_1} \colonequals D_f|_{P_1 \times X}$ is not a sum of two $p$-adic points. Instead we use the explicit Cantor's algorithm \cite{Cantor, SutherlandFastJac} to get a linearly equivalent multiple which does split as a sum of $p$-adic points.

Let $(u_1,v_1)$ be the Mumford representation for $D_{P_1}$. Then using \cite[Algorithm Compose]{SutherlandFastJac} we can compute  $(u_2,v_2)$, the Mumford representation for $2D_{P_1}$. Applying \cite[Algorithm Reduce]{SutherlandFastJac} we obtain the Mumford representation $(u_3, v_3)$ for the reduction of $2 D_{P_1}$ along with $ r =   (y-v_2(x))/ u_3(x)$, satisfying the relationship
\begin{align}
\label{eqn:explicitred}
    2D_{P_1}  = \div (u_1,v_1) = \Div ((y-v_2(x))/u_3(x)) + \div (u_3,v_3).
\end{align}

\begin{remark}
Since the computations for $D_{P_1}$ were done on the regular model, we need to change the equations to the odd degree model. The Mumford divisor for $D_{P_1}$ is a sum of $2$ points over a totally ramified extension of $\Q_p$.  Using the equations \eqref{eqn:modelchange} for the change of model we can map the points to two points $(x_1,y_1),(x_2,y_2)$ on the odd degree model and construct the corresponding degree $2$ Mumford divisor $(u_1,v_1)$ vanishing on the $x$-coordinates using interpolation: 
$u_1(x) = (x-x_1)(x-x_2)$ and $v_1(x) = y_2 \cdot (x-x_1)/(x_2-x_1) + y_1 \cdot (x-x_2)/(x_1-x_2)$.
\end{remark}

Then $2D_{P_1}$ is linearly equivalent to a divisor that splits into a sum of two points over the odd degree model.
The splitting is given by 
\begin{align*}
    \{  Q_1,Q_2 \} \colonequals \{ (469610 \cdot 7 + O(7^9), -15018865 + O(7^9)),\, (499647 + O(7^9),-14480684+ O(7^9)) \}.
\end{align*}
By \eqref{eqn:explicitred} we have
\begin{align*}
    2D_{P_1}=Q_1+Q_2+\Div((y-v_2(x))/u_3(x))+2\infty,
\end{align*}
where 
\begin{align*}
    v_2(x) \colonequals -(462222 + O(7^8))x^3 + (73804 + O(7^8))x^2 + (1999391 + O(7^8))x - 1649234 +O(7^8)
\end{align*}
and 
\begin{align*}
    u_3(x) \colonequals (1 + O(7^8))x^2 + (1977884 + O(7^8))x + 297368\cdot 7 + O(7^8) .
\end{align*}
With the splitting in hand, we can compute $\widetilde{j_b}(P_1)$: 
\begin{align*}
    \frac{1}{2}h_p(P_1-b,2D_{P_1}) + h_p(P_1-b,B-C) =\,&h_p(P_1-b,B-C) +\frac{1}{2} h_p(P_1-b,Q_1+Q_2+2\infty)+ \\
    &\frac{1}{2} h_p(P_1-b,\Div((y-v_2(x))/u_3(x))).
\end{align*}
The divisor $B-C$ is not a sum of points, but we have that $B-C$ is equal to $4\infty_- -\iota b - 5 \iota Q + \Div(g_{P_1})$, where $g_{P_1}$ is given by \eqref{eqn:P1}.
Therefore $\psi(\widetilde{j_b}(P_1))$ is
\begin{align*}
    &h_p(P_1-b,D_{P_1} + B - C)\\ &= \frac12 h_p(P_1-b,Q_1+Q_2+2\infty + 2(4\infty_-  - \iota b - 5 \iota Q))\\ &+\frac12\log( (y-v_2)(P_1-b)/u_3 (P_1-b) )+ \log g_{P_1}(P_1-b).
\end{align*}
Then 
\begin{align*} 
    &\log g_{P_1}(P_1-b) = 6\cdot 7 + 3\cdot 7^2 + 2\cdot 7^3 + 2\cdot 7^4 +O(7^5)\\
    &\log (y-v_2)(P_1-b)/u_3(P_1-b)) = 7^2 + 3\cdot7^3 + 2\cdot7^4 + O(7^5))\\
    &h_p(P_1-b,Q_1+Q_2+2\infty + 2(4\infty_-  - \iota b - 5 \iota Q)) =5\cdot7 + 7^2 + 4\cdot7^3 + O(7^4)
\end{align*}
So $\psi(\widetilde{j_b}(P_1)) =  5\cdot 7 + O(7^2)$.

Now we can calculate $\widetilde{j_b}(P_1)$ in the map $\phi\colon T(\Z_p)_{\widetilde{j_b}(\overline{P})} \to \Z_p^3$ given in Definition~\ref{def:paramT}. We can compute this using the logarithm, normalized by the logarithm at $P$:
\begin{align*}
  &\log(P_0 - b) - \log(P_0 -b) = (0,0), \\
    &\log(P_1 - b) - \log(P_0- b)= (2\cdot 7 + O(7^2), O(7^2)).
\end{align*}

Hence we see $\phi(\widetilde{j_b}(P_0)) = (0,0,6)$ and $\phi(\widetilde{j_b}(P_1)) = (2,0,5)$. By interpolating these values we get \eqref{eqn:jbtildeEx}.

We now discuss the map $\kappa$ using formulas in Section~\ref{sec:integerpoints}.
We will show that the map $\phi \circ \kappa\colon \Z_p^2 \to \Z_p^3$, which is by Proposition~\ref{prop:phicirckappa} given by two homogeneous linear polynomials and one quadratic polynomial, is modulo $p$ equal to
\begin{align}
\label{eqn:kappaex}
(n_1,n_2) \mapsto (n_1, -n_1 - 2n_2, -3n_1^2 - n_1n_2 - n_1 + n_2 - 1).
\end{align}

Following Algorithm~\ref{alg:Qij} we construct the points of $\calM^\times(G_i, f(G_j)) (\Z)$ and $\calM^\times(G_i, c) (\Z)$ for $i, j = 1, 2$ as in \cite[Section~8.3]{EdixhovenLido}.

We work out the example $\calM^\times(G_1, f(G_2)) (\Z)$ here in detail.
Recall from \eqref{eqn:fOfG1G2} that we have $G_1 = [P - \iota P]$ and $f(G_2) = [3P + Q -4\iota P]$. By \eqref{eqn:fiber}, the $\G_m$-torsor $\calM^\times(G_1, f(G_2))$ is $f(G_2)^* \calO_X^\times (G_1)$. Since we want to work with the image in $\calN$, and this representation of $f(G_2)$ is not disjoint from $G_1$ over $\Q$, we represent $G_1$ by the linearly equivalent divisor $\iota b- \infty_+ + \infty_- - Q$ and $f(G_2)$ by the linearly equivalent divisor $3(P-\iota P) + (P - \iota Q)$. These divisors are not disjoint over $\Z$ because $-\iota Q$ and $\iota b$ intersect over $\Z/2\Z$ so
\begin{align*}
    h(P - \iota P, 3(P-\iota P) + (P - \iota Q)) = h_p(\iota b- \infty_+ + \infty_- - Q, 3(P-\iota P) + (P - \iota Q)) + \log(2).
\end{align*}
We can compute
\begin{align*}
    Q_{12} &= ([P - \iota P], [3(P-\iota P) + (P - \iota Q) ], h(\iota b- \infty_+ + \infty_- - Q, 3(P-\iota P) + (P - \iota Q))   \\\notag
    &=(G_1, f(G_1), 5\cdot 7 + 6\cdot 7^2 + 6\cdot 7^3 + O(7^4) ).
\end{align*}

 The remaining $Q_{ij}$ are:
\begin{align*}
    Q_{11}&= (G_1, f(G_1), 2\cdot 7 + 5\cdot 7^3 + O(7^4) ),\notag\\
    Q_{21}&= (G_2, f(G_1),  4\cdot7 + 3\cdot7^2 + 2\cdot7^3 +  O(7^4)),\notag\\
    Q_{22}&= (G_2, f(G_2),3\cdot 7^2 + 4\cdot7^3 +   O(7^4)),\\\notag
    Q_{10}&= (G_1, c,3\cdot 7 + 4\cdot 7^2 + 5\cdot 7^3+   O(7^4)),\\\notag
    Q_{20}&= (G_2, c,2\cdot7 + 2\cdot 7^3 +  O(7^4))\notag.
\end{align*}

\begin{remark}
\label{rem:takeloginN}
In practice, since we will need to add $Q_{ij}$ in $\calN \simeq J(\Q_p) \times J(\Q_p) \times \Q_p$ we use the map $\log\colon J(\Q_p) \to \Q_p^g$ for $i,j = 1, 2$ and for $j=0$, we store $Q_{ij}$ as the vector $(\log(G_i), \log(f(G_j)), h(G_i, f(G_j)))$. This allows us to add in $\Q_p^g$ instead of $J(\Q_p)$.
\end{remark}

We proceed to compute the bijection $\kappa\colon \Z_p^2 \to T(\Z_p)_{\widetilde{j_b}(\overline{P})}$ of the integral points of $T$ modulo $p^2$, as in \cite[Section~8.5]{EdixhovenLido}.
The divisor $j_b(\overline{P}) \in J(\F_p)$ is equal to the image of
\begin{align*}
    \widetilde{G_{t}} \colonequals G_1 + 3 G_2
\end{align*}
in $J(\F_p)$ and correspondingly we define $e_{01} \colonequals 1$ and $e_{02} \colonequals 3$.

Let $\widetilde{G_1}$ and $\widetilde{G_2}$ be a basis for the kernel of reduction  $J(\Z) \to J(\F_p)$. Since
\begin{align*}
    \widetilde{G_1} = -3G_1 + 7G_2, \hspace{.1in} \widetilde{G_2} = 7G_1 + 4G_2
\end{align*}
we define $e_{11} = -3, e_{12} = 7, e_{21} = 7, e_{22} = 4$. 

The map $\kappa_\Z$ is given in coordinates in $\calN$ by sending $(n_1, n_2)$ to
\begin{align*}
&((7 + 7^2 + 7^3 +  O(7^4))\cdot n_1 + (4\cdot7^3 + O(7^4))\cdot n_2 + 5\cdot7 + 5\cdot7^2 + 7^3 + O(7^4), \\ 
&(6\cdot7 + 4\cdot7^2 + 3\cdot7^3 +  O(7^4))\cdot n_1 + (5\cdot7 + O(7^4))\cdot n_2 + 5\cdot7^2 + 3\cdot7^3 +  O(7^4)),\\
 &((6\cdot7 + 5\cdot7^2 + 6\cdot7^3 + O(7^4))\cdot n_1 + (2\cdot7 + 2\cdot7^2 + 6\cdot7^3 +  O(7^4), \\
 &(4\cdot7 + 3\cdot7^2 + 3\cdot7^3 +  O(7^4))\cdot n_1 + (3\cdot7 + 3\cdot7^2 + O(7^4))\cdot n_2 + 4\cdot7 + 2\cdot7^3 +  O(7^4)),\\
 &(4\cdot7 + 6\cdot7^2 + 3\cdot7^3 +  O(7^4))\cdot n_1^2 + (6\cdot7 + 7^2 + 4\cdot7^3 + O(7^4))\cdot n_2^2 +\\
 & (6\cdot7 + 3\cdot7^2 + 2\cdot7^3 + O(7^4))\cdot n_1 + (7 + 7^3 +  O(7^4))\cdot n_2 + 6\cdot7 + 6\cdot7^2 + 3\cdot7^3 +  O(7^4))
\end{align*}
where we apply the logarithm to the first two coordinates as in Remark~\ref{rem:takeloginN}.

Finally, by \cite[Theorem~4.10]{EdixhovenLido}, the map $\kappa_\Z$ extends to a bijection 
\begin{align}
\label{def:kappaExample}
    \kappa\colon \Z_p^2 \to T(\Z_p)_{\widetilde{j_b}(\overline{P})}
\end{align}
with image $\overline{T(\Z)}_{\widetilde{j_b}(\overline{P})}$. 
This map $\phi \circ \kappa$ is polynomials $(\kappa_1,\kappa_2,\kappa_3) \in \Q_p[x_1,x_2] ^3$, with $\kappa_1,\kappa_2$ homogeneous linear and $\kappa_3$ at worst quadratic. Applying Corollary~\ref{lem:paramT}, we obtain the formula for $\phi \circ \kappa$ given in \eqref{eqn:kappaex}.

We now have the tools to prove the upper bound on the number of points in the residue disk $\#X(\Z)_{\overline{P}}$. We define 
\begin{align*}
    \overline{p_1} \colonequals (\varphi\circ\kappa)^* \overline{g_1} = -n_1 - 2n_2, \hspace{.1in} \overline{p_2} \colonequals (\varphi\circ\kappa)^* \overline{g_2} = n_1^2 - 2n_1n_2 - n_1 + 2n_2,
\end{align*}
and $\overline{A} \colonequals \F_p[n_1,n_2]/(\overline{p_1},\overline{p_2})$. The ring $\overline{A}$ is isomorphic to $\F_p[n_2]/(n_2^2 - 3n_2) \simeq \F_p \times \F_p$, so by \cite[Theorem~4.12]{EdixhovenLido} we have an upper bound of $2$ on $\#X(\Z)_{\overline{P}}$. Specifically, we see that there is at most one point reducing to $P_0$, namely $P$ itself, and at most one point reducing to $P_4$ in $X(\Z/p^2\Z)_{\overline{P}}$; the other $P_\nu$ have no rational points lying over them.

\begin{remark}
\label{rem:gccubic}
If we calculate $\kappa$ and $\widetilde{j_b}$ with greater $p$-adic precision, we can compute the point reducing to $P_4$ with greater precision. This can be done by brute force, that is, trying all lifts of the found solution $n_1 = 1, n_2 = 3, \nu = 4$ and seeing when any of the calculated values of $\kappa$ or $\widetilde{j_b}$ agree modulo the required precision. However, there is a more efficient way. We can look at the ``higher residue disks''  $X(\Z_p)_{P_4}$ and $T(\Z_p)_{\widetilde{j_b}(P_4)}$, consisting of points that reduce to a specified $\Z/p^2\Z$-point. We can parametrize $X(\Z_p)_{P_4}$ with the map $\Z_p \to X(\Z_p)_{P_4}$ sending $\mu$ to $P_{4 + p\mu}$. With respect to our usual map $\phi: T(\Z_p)_{\widetilde{j_b}(\overline{P})} \to \Z_p^3$, we get a bijection of the higher residue disk of the torsor $T(\Z_p)_{\widetilde{j_b}(P_4)} \to (1,0,2) + p\Z_p^3$. Given these identifications, the inclusion $\widetilde{j_b}\colon  X^\sm(\Z_p)_{P_4} \to T(\Z_p)_{\widetilde{j_b}(P_4)}$ is given by power series that are linear modulo $p$. Like in Section~\ref{sec:embeddingcurve}, these can be found by interpolation. Similarly, $\kappa$ restricted to $(1+p\Z_p) \times (3 + p\Z_p)$ gives the inclusion $\kappa\colon  \overline{T(\Z)}_{\widetilde{j_b}(P_4)} \to T(\Z_p)_{\widetilde{j_b}(P_4)}$. For these identifications, $\kappa$ is actually homogeneous linear modulo $p$. Solving the resulting affine linear system of equations, we get that the only possible intersection of the image of $\kappa$ and of $\widetilde{j_b}$ in the higher residue disk $T(\Z/p^3\Z)_{\widetilde{j_b}(P_4)} \simeq \F_p^3$ is $(5,1,5)$, corresponding to $P_{4 + p\mu}$ with $\mu = 4$. This is the point $P_{32} \in X(\Z/p^3\Z)_{P_4}$.

In total, we can strengthen Proposition~\ref{prop:points} to say the residue disk $X(\Z)_{\overline{P}}$ is contained in the set
\begin{align*}
    \{P,(4\cdot 7 + 4\cdot 7^2 + O(7^3), 6 + 6\cdot 7 + 6 \cdot 7^2 + O(7^3))\}. 
\end{align*}
\end{remark}

\section{Acknowledgements}
This project began at the 2020 Arizona Winter School on Nonabelian Chabauty. It is a pleasure to thank the project supervisors Bas Edixhoven, Guido Lido, and Jan Vonk for their help. We are also very thankful to Jennifer Balakrishnan, Steffen M\"uller, John Voight, David Holmes, Alexander Betts, Edgar Costa, Jeroen Sijsling, Andrew Sutherland and Amnon Besser for helpful correspondence or computational advice.  We thank the anonymous referee for their careful feedback, especially regarding the proof of the main theorem.

\appendix

\section{Equations}
We provide the equations used in the computations of Section~\ref{sec:example}.

We give coordinates $((x,y),(u,v))$ to $X\times X$. With this notation, the equations that define the divisor $D_f$ are the following.
\begin{footnotesize}
\begin{align}
\label{eqn:DfExample}
    D_f\colonequals&[x^5 - x^3y - xy - y^2 - x - y,
    \\\notag&
    u^5 - u^3v - uv - v^2 - u - v,\\\notag &
    1120x^{20}u^4 - 2068x^{20}u^3 + 8124x^{19}u^4 + 2407x^{20}u^2 - 16894x^{19}u^3 + 
    35279x^{18}u^4 - 1641x^{20}u + \\\notag &
    18092x^{19}u^2 - 67012x^{18}u^3 +
    102591x^{17}u^4 + 378x^{20} - 8178x^{19}u +
    58447x^{18}u^2 - 173283x^{17}u^3 + \\\notag &
    216476x^{16}u^4 + 774x^{19} - 14247x^{18}u + 103331x^{17}u^2 -
    297137x^{16}u^3 + 334741x^{15}u^4 + 1458x^{18} -\\\notag & 31130x^{17}u + 180514x^{16}u^2 - 358567x^{15}u^3 +
    360468x^{14}u^4 + 10605x^{17} - 90380x^{16}u +
    290195x^{15}u^2 - \\\notag&
    395289x^{14}u^3 + 240873x^{13}u^4 + 20415x^{16} - 159334x^{15}u + 394529x^{14}u^2 - 407100x^{13}u^3 + 44248x^{12}u^4 +
    \\\notag & 22701x^{15} -
    112959x^{14}u + 418497x^{13}u^2 - 493887x^{12}u^3 - 105112x^{11}u^4 + 25606x^{14} - 115611x^{13}u +\\\notag &
    111265x^{12}u^2 - 417580x^{11}u^3 - 92961x^{10}u^4 + 1092x^{13} - 103527x^{12}u + 145152x^{11}u^2 -
    88490x^{10}u^3 - \\\notag &
    92811x^9u^4 + 48856x^{12} + 186438x^{11}u + 267721x^{10}u^2 - 155622x^9u^3 -
    45395x^8u^4 - 27776x^{11} - \\\notag &
    191295x^{10}u - 178159x^9u^2 - 70489x^8u^3 + 16905x^7u^4 -
    61956x^{10} - 74059x^9u + 378244x^8u^2 +\\\notag & 232801x^7u^3 + 
    15979x^6u^4 + 74366x^9 +
    338472x^8u + 227589x^7u^2 - 74613x^6u^3 - 16012x^5u^4 - \\\notag &
    87675x^8 - 182672x^7u -
    189206x^6u^2 + 26802x^5u^3 + 25133x^4u^4 - 85989x^7 - 42976x^6u + 119160x^5u^2 +\\\notag &
    38380x^4u^3 - 14569x^3u^4 + 57369x^6 + 50376x^5u - 22878x^4u^2 - 26236x^3u^3 +
    5653x^2u^4 - 19638x^5 -\\\notag &
    66959x^4u + 10199x^3u^2 + 7737x^2u^3 - 1185xu^4 - 
    18109x^4 + 33891x^3u - 10338x^2u^2 + 
    126xu^3 + \\\notag &
    90u^4 + 8894x^3 - 13882x^2u + 3365xu^2 -    189u^3 - 1493x^2 + 903xu - 105u^2 - 176x + 18u + 4,\\
    & \notag
    7605023584402176072496x^8u^2 + 276848668324194788374x^8u + 2162467398048698636700x^7u^2 -\\ \notag & 6272554892698832692599x^6yu^2 - 4626446567682633747828x^8v - 1168446771586826201673x^8 -\\ \notag & 9165162915676858733619x^7u + 2241777840578137196064x^6yu - 8418141092008037071834x^6u^2 -\\\notag & 13292836185052144419762x^5yu^2 + 754031123597981360894x^7v + 6328906343710703634915x^6yv +\\\notag & 2615195628519325252191x^7 + 1831262799801461507208x^6y + 2756070458250784948869x^6u +\\\notag &
    15428857376010803153841x^5yu - 11784051570902048135703x^5u^2 - 7230872538984499657093x^4yu^2 +\\\notag & 16912156368781966844899x^6v + 8794134244461097697655x^5yv + 13382241469127150196465x^6 +\\\notag & 4082469582390924565047x^5y + 21852540598540798087489x^5u + 13245519579554143163167x^4yu -\\\notag & 22985066915160029536074x^4u^2 - 23255128704790712417887x^3yu^2 + 13682822171560412185605x^5v -\\\notag & 165783020433170604550x^4yv - 6931902302166164206278x^5 - 5083451259029072420619x^4y -\\\notag & 11826350429569203951840x^4u - 19199699515311811452213x^3yu - 28484484698745046075669x^3u^2 -
    \\\notag & 17690076715222265602489x^2yu^2 + 17805473443696348827856x^4v + 675202808346140479378x^3yv +\\\notag & 6675814886892603310402x^4 + 5577161777751351740903x^3y + 19969878692979973055652x^3u +\\\notag & 18120117063433135735083x^2yu + 936713375105971953531x^2u^2 + 11466853454037386066020xyu^2 +\\\notag & 10542523972242190209720x^3v + 8824421921807720328364x^2yv + 11877160806671853672804x^3 +\\\notag & 13363913247174903062953x^2y + 14059453652617340471247x^2u + 10218057833893227356605xyu -\\\notag & 308361787245220032444xu^2 - 5322779956111165805354yu^2 + 5505912629321680476560x^2v -\\\notag & 4290695327689320279111xyv - 7612900075627672207215x^2 - 14312446660999532149696xy +\\\notag & 4434640084437900284987xu + 3704885128833955385271yu - 993796068912520397282u^2 +\\\notag & 57535042100777081983xv + 3829830430486931582408yv + 5885803647094172346013y +\\\notag & 960790192851544016507u + 281506727438003913980v + 113825130829311801917,\\\notag &
    790135714013668417211x^8u^2 - 52199251698889313788x^8u + 445626397822123380960x^7u^2 -\\\notag & 484065148072652139393x^6yu^2 - 355589770017865569639x^8v - 97839554801178078020x^8 -\\\notag & 678398566039036992539x^7u + 155198586393263487818x^6yu - 113052264818131543479x^6u^2 -\\\notag & 874765196307671212424x^5yu^2 + 50893236050896468243x^7v + 549806068461932423405x^6yv +\\\notag & 245852373764948827027x^7 + 222973665578085376766x^6y - 186006391998859651031x^6u +\\\notag & 918135020900189841469x^5yu - 523150712434256670561x^5u^2 - 328927822772590067729x^4yu^2 +\\\notag & 1388867642711454788442x^6v + 882684613081080621057x^5yv + 1142791546745352334216x^6 +\\\notag & 533732004549278022010x^5y + 394464353147344850914x^5u + 874586564270896523236x^4yu -\\\notag & 1503623861758469781638x^4u^2 - 1118256877330123036794x^3yu^2 + 962253617070423872834x^5v +\\\notag & 260675420287904377496x^4yv - 73108557049802456668x^5 - 177841514864980758518x^4y -\\\notag & 1357965873921914116106x^4u - 1595337468013963640622x^3yu - 1882558303840937888797x^3u^2 -\\\notag & 1293922634022119677492x^2yu^2 + 1390753692690189767706x^4v + 246438908010171275168x^3yv +\\\notag & 793691222208583979104x^4 + 499223278514256382778x^3y + 645256167770372257021x^3u +\\\notag & 984786145000107598929x^2yu - 280718524673749556697x^2u^2 + 779933023636684223799xyu^2 +\\\notag & 842189446494471065427x^3v + 558551444022004233780x^2yv + 913241896994237593431x^3 +\\\notag & 1244963363551342949690x^2y + 727117765460043207926x^2u + 1012441030923028187282xyu -\\\notag & 21753359867708939458xu^2 - 344942106360625888966yu^2 + 353025200232170583936x^2v -\\\notag & 211033121948623991455xyv - 163875785683850219832x^2 - 617198754625174179093xy +\\\notag & 597830134728356122829xu + 169901861802716830954yu - 82203224665107226192u^2 +\\\notag & 82310455430799619016xv + 191169787322405231086yv + 341475392487935405751x +\\\notag & 350318508927358217032y - 21028731891073941584u - 9558514495942700720v].
\end{align}
\end{footnotesize}

Now we assume that $u$ and $v$ are elements of the function field of $X$ satisfying $v^2 + (u^3 +u+1) = u^5 - u$. The equation $g_{P_0}$ is given by
\begin{align}
\label{eqn:P0}
g_{P_0} \colonequals& 118016503u^{11} + 793929202u^{10} - 2478346563u^9 - 3325919630u^8 - \\\notag
   & 3561952636u^7 + 2886039937u^6 + 5879367604u^5 - 3830171961u^4 + \\\notag
    &75101411u^3 + 2188669692u^2 - 697370245u + 85830184)/(338078160u^{14} + \\\notag
    &1369216548u^{13} + 2510230338u^{12} + 2713077234u^{11} + 1318504824u^{10} - \\\notag
    &3414589416u^9 - 135231264u^8 - 236654712u^7 - 6668591706u^6 + \\\notag
   & 1850977926u^5 + 3220194474u^4 - 1293148962u^3 + 397241838u^2 - \\\notag
    &8451954u)v + (375507055u^{14} + 718827791u^{13} - 1351825398u^{12} - \\\notag
    &3390292268u^{11} - 6705483125u^{10} + 42915092u^9 - 3840900734u^8 - \\\notag
    &10868247049u^7 + 12659952140u^6 + 12198614901u^5 - 5503860549u^4 - \\\notag
    &1083606073u^3 + 1748789999u^2 - 686641472u + 85830184)/(338078160u^{14} + \\\notag
    &1369216548u^{13} + 2510230338u^{12} + 2713077234u^{11} + 1318504824u^{10} - \\\notag
    &3414589416u^9 - 135231264u^8 - 236654712u^7 - 6668591706u^6 + \\\notag
    &1850977926u^5 + 3220194474u^4 - 1293148962u^3 + 397241838u^2 - \notag
    8451954u).
\end{align}

Similarly, the equation $g_{P_1}$ is given by
\begin{align}
\label{eqn:P1}
g_{P_1} &\colonequals (9192u^{12} + 11490u^{11} + 10341u^{10} + 104559u^9 + 116049u^8 + 189585u^7 + \\\notag
   & 24129u^6 - 659526u^5 - 335508u^4 + 291846u^3 + 135582u^2 + 34470u + \\\notag
   & 1149)/(17360u^{11} + 35588u^{10} + 40362u^9 + 23002u^8 - 18662u^7 -\\\notag
   & 161014u^6 + 333746u^5 - 518630u^4 + 361088u^3 - 108500u^2 + 21266u -\\\notag
   & 434)v + (-9192u^{14} - 2298u^{13} - 8043u^{12} - 118347u^{11} - 181542u^{10} -\\\notag
   & 351594u^9 - 2298u^8 + 689400u^7 - 13788u^6 - 476835u^5 + 65493u^4 +\\\notag
   & 167754u^3 + 52854u^2 - 13788u + 2298)/(17360u^{11} + 35588u^{10} +\\\notag
   & 40362u^9 + 23002u^8 - 18662u^7 - 161014u^6 + 333746u^5 - 518630u^4 +\\\notag
   & 361088u^3 - 108500u^2 + 21266u - 434).\notag
\end{align}

\newcommand{\etalchar}[1]{$^{#1}$}

\end{document}